\documentclass[11pt,reqno]{amsart}

\usepackage{amssymb}
\usepackage{amscd}
\usepackage{amsfonts}
\usepackage{mathrsfs}
\usepackage{setspace}
\usepackage{version}
\usepackage{mathtools}
\usepackage[english]{babel}

\numberwithin{equation}{section} 
\theoremstyle{plain}
\newtheorem{theorem}[subsection]{Theorem}
\newtheorem{lemma}[subsection]{Lemma}
\newtheorem{definition}[subsection]{Definition}
\newtheorem{remark}[subsection]{Remark}
\newtheorem{proposition}[subsection]{Proposition}
\newtheorem{corollary}[subsection]{Corollary}

\renewcommand{\leq}{\leqslant}
\renewcommand{\geq}{\geqslant}
\newcommand{\n}{n}
 \newcommand{\m}{m}
 \newcommand{\1}{1}

\newcommand\R{\mathbb{R}}
\newcommand\N{\mathbb{N}}

\newcommand\Q{\mathbb{Q}}

\begin{document}

\title[Stable compactness]{On compactness in $L^0$-modules}

\author{Asgar Jamneshan}
\address{Department of Mathematics and Statistics, University of Konstanz}
\email{asgar.jamneshan@uni-konstanz.de}

\author{Jos\'e M.~Zapata}
\address{Departamento de Matem\'aticas, Universidad de Murcia}
\email{jmzg1@um.es}

\thanks{The authors would like to thank Michael Kupper for advice}
\thanks{Asgar Jamneshan is supported by DFG project KU-2740/2-1. Jos\'e M.~Zapata is supported by the grant MINECO MTM2014-57838-C2-1-P}


\begin{abstract} 
Several results in functional analysis are extended to the setting of $L^0$-modules, where $L^0$ denotes the ring of all measurable functions $x\colon \Omega\to \mathbb{R}$. 
The focus is on results involving compactness. 
To this end, a notion of  \emph{stable compactness} is introduced, and it is argued that the conventional notion of compactness does not allow to establish a functional analytic discourse in $L^0$-modules. 
Several characterizations of stable compactness are discussed, and its importance in applications is highlighted. 

\end{abstract}

\maketitle

\setcounter{tocdepth}{1}

\section{Introduction}

Let $(\Omega,\mathcal{F},\mathbb{P})$ be a probability space.
Let $L^0=L^0(\Omega,\mathcal{F},\mathbb{P})$ denote the space of all real-valued Borel measurable functions modulo almost everywhere identity. 
The \emph{locally $L^0$-convex topology}  on $(L^0)^d, d\geq 1$, is defined through neighbourhoods of the form $\{x\in (L^0)^d\colon \|x\|<r\text{ a.e.}\}$ where $\|x\|:=(\sum_{i=1}^d x_i^2)^{1/2}$ is an $L^0$-valued norm and $r\in L^0$ with $r>0$ a.e., see \cite{Cheridito2012,kupper03}. 
The space $(L^0)^d$ is a topological $L^0$-module of rank $d$. 
By verifying that $(L^0)^d$ is anti-compact\footnote{A topological space $X$ is called anti-compact if any compact set in $X$ is finite.} in general,  
we confirm the observation in the literature (\cite{Cheridito2012,DJKK13,guo08}) that there is no hope to extend theorems in analysis to $(L^0)^d$ with a classical notion of compactness, for example a Heine-Borel theorem. 
More generally, we prove anti-compactness for arbitrary stable topological $L^0$-modules. 
In \cite{DJKK13}, a notion of conditional compactness is introduced within the conditional set theory, and it was shown that it allows to extend results to the setting of conditional locally convex topological vector spaces. 
The interpretation of conditional compactness in $L^0$-modules is called \emph{stable compactness} which seems to be a suitable substitute of classical compactness in modules, as it permits to extend a thorough functional analytic discourse to the setting of $L^0$-modules which is illustrated in this article. 
Roughly speaking, a subset of a stable topological $L^0$-module is said to be stable compact if for every open covering $\mathcal{O}$ of $S$ there exists a measurable countable partition $(A_k)$ of $\Omega$ such that $S$ is covered by finitely many elements of $\mathcal{O}$ on each $A_k$.  
In the finite dimensional case $(L^0)^d$ characterizations of stable compactness exist in the literature, see \cite{Cheridito2012}, which are applied to solve stochastic control problems in \cite{ch11,kupper08,JKZ2017control} and to establish a Brouwer fixed point theorem in \cite{DKKM13}. 
We will discuss further equivalent characterizations of stable compactness. 
One of them builds a unique correspondence to measurable selection theory, and thus closes one gap between $L^0$-module theory and probabilistic analysis theory (see e.g.~\cite{schweizer2011probabilistic}) and random set theory (see e.g.~\cite{molchanov2005theory}). 

For an arbitrary $L^0$-module $E$, a locally $L^0$-convex topology can be defined by a family $\mathscr{P}$ of $L^0$-seminorms, see \cite{kupper03}. 
Similarly as in topological vector spaces, one introduces neighbourhoods of the form $\{x\in E\colon \sup_{p\in N} p(x)<r\}$ where $N\subset \mathscr{P}$ is finite. 
As shown in \cite{kupper03}, the locally $L^0$-convex topology allows to extend the Hahn-Banach extension and separation theorems and the Fenchel-Moreau theorem to $L^0$-modules.  
The $L^0$-scalar multiplication defines a function $\mathcal{F}\times E\to E$ given by $(A,x)\mapsto 1_A x$. 
It is said that $E$ is stable under countable pastings, or stable for short, if for every measurable countable partition $(A_k)$ of $\Omega$ and every sequence $(x_k)$ in $E$ there exists a unique $x\in E$ such that $1_{A_k} x = 1_{A_k} x_k$ for all $k$. 
We call a module topology $\mathcal{T}$ on $E$ stable whenever $E$ is stable and $\mathcal{T}$ admits a topological base consisting of a stable collection of stable subsets of $E$, see \cite{DJKK13}.  
The hyperplane separation theorem in \cite{kupper03} is proved under the assumption that the inducing family $\mathscr{P}$ of $L^0$-seminorms is closed under finite suprema and countable pastings, which is equivalent to the property that  the locally $L^0$-convex topology on $E$ induced by $\mathscr{P}$  is in fact a stable topology, see the discussion in Section \ref{s:sec2} for details.  
The abstraction of stability leads to the notion of a \emph{conditional set} in \cite{DJKK13}. 
As a result, all objects in conditional set theory satisfy stability under pastings.\footnote{Conditional set theory is formalized in the context of arbitrary complete Boolean algebras, so that pastings might be uncountable.} 
Conditional set theory provides a formalism to establish systematically conditional versions of results in linear algebra, topology, measure theory, functional analysis, etc., see  \cite{DJKK13,jamneshan13,jamneshan2017measures,martindiss,OZ2017stabil,zapata2016eberlein}. 
It follows from the definition of a conditional set that there is a close connection between classical structures and corresponding conditional structures. 
This connection is systematically studied in \cite{DJKK13}, and is reflected in all definitions, propositions and theorems there. 
In particular, a one-to-one correspondence between conditional topological base, conditional convergence, and conditional continuity and respective classical structures is established, see \cite[Section 3]{DJKK13}, and more specifically \cite[Propositions 3.5, 3.11, 3.22]{DJKK13}.  
Moreover, in the context of the associated measure algebra, an isometric isomorphism between conditional real numbers and $L^0$ is proved in \cite[Theorem 4.4]{DJKK13}, which in particular implies that it follows from the definition of a conditional vector space that its underlying classical structure is a stable $L^0$-module, see \cite{DJKK13}.  
We verify that this correspondence is an equivalence of categories, which enables to study the classical meaning of results in conditional set theory in $L^0$-modules. 
We start by studying algebraic properties in Section \ref{s:sec1}, where we argue that the conventional notion of an algebraic basis is limited to finitely ranked $L^0$-modules. 
We suggest a notion of stable basis, obtained as an interpretation of a vector space basis in conditional set theory, to work with in infinite dimensional settings. 

A different type of topology which is well studied in the context of $L^0$-modules is the \emph{$(\epsilon,\lambda)$-topology} or \emph{$L^0$-topology} which can be viewed as a generalization of the topology of convergence in probability, see  \cite{guo2009separation,haydon1991randomly}.  
Several results in functional analysis were extended to the setting of $L^0$-modules endowed with an $(\epsilon,\lambda)$-topology, see \cite{guo08,guoshi11,guo2009separation} and the references therein.
Using the conventional notion of compactness, it is proved in \cite{guo08} that the Banach-Alaoglu theorem holds in an $L^0$-module $E$ endowed with an $(\epsilon,\lambda)$-topology if and only if $E$ has an essentially atomic support.   
We prove that if $E$ is endowed with a stable topology then it is anti-compact if and only if the support of $E$ is atomless. 
However, stable compactness allows to establish an extension of the the Banach-Alaoglu theorem to $L^0$-modules  in full generality. 
Moreover, we provide, among others, an $L^0$-module version of the Heine-Borel theorem, Tychonoff's theorem, the Eberlein-\v{S}mulian theorem, and the James' compactness theorem.  
Further results in functional analysis are extended to $L^0$-modules as well, e.g.~the Bipolar theorem, the Baire category theorem, the uniform boundedness principle, Banach's fixed point theorem, and the \`{A}rzela-Ascoli theorem. 
In \cite{guo10}, connections between the locally $L^0$-convex topology and the $(\epsilon,\lambda)$-topology are established, see also \cite{guo2015random1,guo2015random2}.  
In Section \ref{s:sec2} we provide an analysis of the relations between the stable topology, the locally $L^0$-convex topology and the $(\epsilon,\lambda)$-topology, which combined with the results in \cite{guo10}, allow to derive  extensions of functional analytic results with respect to all three types of topologies in the last two sections. \\

\noindent \textbf{Related literature.}
The $L^0$-topology in a module setting is introduced in \cite[Section 5]{haydon1991randomly} to study the local behaviour of Bochner-Lebesgue spaces. 
Motivated by problems in probabilistic analysis, the  $(\epsilon,\lambda)$-topology is introduced in $L^0$-modules, see the survey in \cite{guo2015random1,guo2015random2} for a historical background, and see \cite{schweizer2011probabilistic} for an introduction to probabilistic analysis.  
The locally $L^0$-convex topology is introduced in \cite{kupper03}, motivated by problems in mathematical finance, see \cite{Cheridito2012,DKKM13,kupper04} for further results in this direction. 
The stable topology is a result of conditional set theory and the connections between classical and conditional structures in \cite{DJKK13}. 
Conditional set theory is closely related to Scott-Solovay-Vop\v{e}nka's Boolean-valued models of set theory, see \cite{bell2005set} for an introduction, and to toposes of sheaves on complete Boolean algebras, see  \cite{maclane2012sheaves} for an introduction. 
In \cite{jamneshan13,jamneshan2014topos}, the connection between conditional sets and toposes of sheaves is studied. 
The term Boolean-valued analysis refers to the application of Boolean-valued models to the standard model of ZFC, see for a thorough account the excellent monographs \cite{kusraev2012boolean,kutateladze2012nonstandard} and their extensive list of references. 
In \cite{kusraev2012boolean} the connection between vector spaces in a Boolean-valued model and modules in the standard model are investigated systematically. 
Conditional analysis is applied to stochastic control problems, dynamic risk sharing, representation of preferences, duality of conditional risk measures, equilibrium pricing, Principal-Agent problems and vector duality in  \cite{BH14,martinsamuel13,kupper08,ch11,cs12,DJ13,drapeau2017fenchel,kupper11,frittelli14,JKZ2017control,zapata16,OZ2017stabil}.

The rest of this paper is organized as follows. Section \ref{s:notation} fixes the notations and introduces the preliminaries. 
In Section \ref{s:sec1} algebraic aspects of $L^0$-modules are discussed. In Section \ref{s:sec2} the connection among the three types of $L^0$-module topologies and conditional locally convex topological vector spaces is analyzed. 
In Section \ref{s:compact} stable compactness is introduced and characterized. 
In Section \ref{s:funcana} several functional analytic results are extended to $L^0$-modules and stable metric spaces for different types of module topologies.   

\section{Preliminaries}\label{s:notation}

We denote by boldface letters objects in conditional set theory and with non-boldface ones those in classical set theory. 
Throughout we fix a probability space $(\Omega,\mathcal{F},\mathbb{P})$, unless we mention otherwise. 
Measurable sets in $\Omega$ are denoted by the upper case latin letters $A,B,C$ and in their indexed forms. 
We will always identify two functions on $\Omega$ or measurable sets in $\Omega$ if they are equal almost everywhere. 
In the latter case this identification leads to the associated measure algebra which is denoted by $\mathcal{A}=(\mathcal{A},\wedge,\vee,{}^c,1,0)$, see \cite[Chapter 31]{halmos09} for an introduction to measure algebras. 
We denote the equivalence classes in $\mathcal{A}$ by the lower case latin letters $a,b,c$ and in their indexed forms. 
If we use both notations $A$ and $a$, $B$ and $b$, etc., then it is always understood that $A$ is a representative of $a$, $B$ is a representative of $b$, etc. 
Let $p(a)$, $a\in \mathcal{A}$, denote the set of all partitions of $a$, i.e.~a family $(a_i)$ consisting of pairwise disjoint elements in $\mathcal{A}$ with $\vee a_i=a$. 
Recall the following two facts which are crucial in the following analysis. 
\begin{itemize}
\item[(i)] $\mathcal{A}$ is a complete Boolean algebra. 
\item[(ii)] $\mathcal{A}$ satisfies the countable chain condition, that is, all partitions  are at most countable. 
\end{itemize}
An important property will be a stability property w.r.t.~countable gluings which we will impose on sets, functions, sequences, modules, topologies, filters, etc. 
We name this property \emph{stable under countable concatenations}, or \emph{stable} for short, e.g.~a stable set, a stable module, a stable topology, etc. 
The essence of this stability property is a conditional set, see \cite{DJKK13} for an introduction to conditional set theory. 
In the following, basic concepts in conditional set theory are summerized.  
\begin{definition}
Let $X$ be a non-empty set. 
A collection $\mathbf{X}$ of objects $x|a$ for  $a\in\mathcal{A}$ and $x\in X$ is said to be a \emph{conditional set with carrier $X$}, if 
\begin{itemize}
\item[(i)]  $x|a=y|b$ implies $a=b$, 
\item[(ii)] \emph{Consistency}: $x|a=y|a$ and $b\leq a$ imply  $x|b=y|b$,\footnote{In every Boolean algebra, $b\leq a$ is defined by $b=b\wedge a$.}
\item[(iii)] \emph{Stability}: for all $(a_k)\in p(1)$ and every countable family $(x_k)$ in $X$ there exists a unique $x\in X$ such that $x|a_k=x_k|a_k$ for all $k$. 
We call $x$ the \emph{concatenation of $(x_k)$ along $(a_k)$}, and write the formal expression $x=\sum x_k|a_k$. 
\end{itemize}
\end{definition}

\begin{example}
Let $X$ be a non-empty set. 
We denote by $L^0_s(X)$ the set of all \emph{step functions} $x=\sum_k 1_{A_k} x_k\colon \Omega\to X$ where $(a_k)\in p(1)$ and $\sum_k 1_{A_k} x_k$ denotes the function with value $x_k$ on $A_k$ for all $k$. 
Identify on $L^0_s(X)\times\mathcal{A}$ two pairs $(x,a)$ and $(y,b)$ whenever $a=b$ and $x(\omega)=y(\omega)$ for almost all $\omega\in A$. 
Denote by $x|a$ the equivalence class of $(x,a)$. 
Then the collection $\mathbf{X_s}$ of all $x|a$, $a\in\mathcal{A}$ and $x\in L^0_s(X)$, is a conditional set with carrier $L^0_s(X)$. 
The set of step functions $L^0_s(\mathbb{N})$ is the carrier of the \emph{conditional natural numbers} $\mathbf{N_s}$.  
\end{example}

Let $\mathbf{X}$ be a conditional set with carrier $X$. 
A subset $Y\subset X$ is said to be \emph{stable}, if $Y\neq \emptyset$ and $\sum y_k|a_k\in Y$ for all $(y_k)\subset Y$ and $(a_k)\in p(1)$. 
A stable subset $Y\subset X$ induces the conditional set $\mathbf{Y}:=\{y|a\colon y\in Y, \, a\in\mathcal{A}\}$ which is called a \emph{conditional subset} of $\mathbf{X}$. 
Let $Y\subset X$ be non-empty. Then 
\[
\text{st}(Y):=\left\{\sum y_k|a_k\colon (y_k) \subset Y, \, (a_k)\in p(1)\right\}. 
\]
is a stable set which is called the stable hull of $Y$. 
For a countable family $(Y_k)$ of stable subsets of $X$ and $(a_k)\in p(1)$, define the stable set
\[
\sum Y_k|a_k:=\left\{\sum y_k|a_k\colon y_k\in Y_k \text{ for each }k\right\}. 
\]
A collection $\mathscr{Y}$ of stable subsets of $X$ is said to be a \emph{stable collection of stable subsets}, if  $\mathscr{Y}\neq \emptyset$ and $\sum Y_k|a_k\in \mathscr{Y}$ for all $(Y_k)\subset \mathscr{Y}$ and $(a_k)\in p(1)$. 
Similarly as above,  a stable collection $\mathscr{Y}$ of stable sets induces a conditional collection of conditional sets, see \cite[Section 2]{DJKK13}. 
For a non-empty collection $\mathscr{Y}$ of subsets of $X$, we form its stable hull by 
\[
\text{st}(\mathscr{Y}):=\left\{\sum \text{st}(Y_k)|a_k\colon (Y_k) \subset \mathscr{Y}, \, (a_k)\in p(1)\right\}. 
\]
Given a non-empty family $(\mathbf{X}_i)$ of conditional sets, the direct product $\prod X_i$ of the family $(X_i)$ of carriers defines in a natural way a conditional set which is called the \emph{conditional Cartesian product} of the family $(\mathbf{X}_i)$. 
Let $\mathbf{X}$ and $\mathbf{Y}$ be conditional sets. 
A relation $S\subset X\times Y$ is said to be \emph{stable}, if $(\sum x_k|a_k,\sum y_k|a_k)\in S$ whenever $(x_k, y_k)\in S$ for all $k$ and $(a_k)\in p(1)$. 
A \emph{conditional relation} $\mathbf{S}$ is the conditional set induced by a stable relation $S$. 
A function $f\colon X\to Y$ is said to be \emph{stable}, if $f(\sum x_k|a_k)=\sum f(x_k)|a_k$ for all $(x_k)\subset X$ and $(a_k)\in p(1)$. 
A \emph{conditional function}  $\mathbf{f\colon X\to Y}$ is the conditional set $\mathbf{G_f}$ induced by the graph $G_f$ of a stable function  $f\colon X\to Y$. 
A family $(x_i)=(x_i)_{i\in I}$ of elements in $X$ is said to be \emph{stable}, if it is parametrized by a stable function $I\to X$.\\

Let $(\mathbb{R},+,\cdot,\leq)$ be the totally ordered field of real numbers. 
Let $L^{0}$ denote the space of all real-valued Borel measurable functions on $(\Omega,\mathcal{F},\mathbb{P})$. 
Consider on $L^0$ the order of almost everywhere dominance. 
We will always understand equalities and inequalities between measurable functions in the almost everywhere sense. 
Further, let $L^0_+:=\{r\in L^0: r\geq 0\}$ and $L^0_{++}:=\{r\in L^0: r> 0\}$. 
Recall that $(L^{0},+,\cdot,\leq)$ is a Dedekind complete Riesz algebra. 
For a subset $X\subset L^0$, we write $\sup X=\text{ess$\,$sup} \,X$ and $\inf X= \text{ess$\,$inf} \,X$ whenever these quantities exist. 
Define on $L^0\times \mathcal{A}$ the equivalence relation $(r,a)\sim (t,b)$, if $a=b$ and $1_A r=1_A t$, and denote by $r|a$ the equivalence class of $(r,a)$. 
Then the collection $\mathbf{R}$ of all equivalence classes $r|a$ is a conditional set with carrier $L^0$, where the concatenation of $(x_k)$ along $(a_k)$ corresponds to $\sum_k 1_{A_k} x_k:=1_{A_1} x_1  + 1_{A_2} x_2  +\ldots$. 
Addition and multiplication are stable functions which define the structure of a conditional field on $\mathbf{R}$, and the order of almost everywhere dominance is a stable relation which defines a conditional total order on $\mathbf{R}$ such that $(\mathbf{R},+,\cdot,\leq)$ is a conditional totally ordered field, see \cite[Section 4]{DJKK13}. 
In this context, $L^0_+$ and $L^0_{++}$ correspond to the conditional set of positive and strictly positive conditional real numbers, respectively. 

Recall that $L^0_s(\mathbb{N})$ is the carrier of the conditional natural numbers. 
We understand $\mathbb{N}$ as a subset of $L^0_s(\mathbb{N})$ via the embedding $n\mapsto 1_\Omega n$. 
A stable family $(x_n)=(x_n)_{n\in L^0_s(\mathbb{N})}$ parametrized by a stable function $L^0_s(\mathbb{N})\to X$ is called a \emph{stable sequence}. 
For each $\n\in L^0_s(\mathbb{N})$, define the stable set 
\[
\{1\leq m\leq n\}:=\{m\in L^0_s(\mathbb{N})\colon 1\leq m\leq n\}. 
\]
A stable set $X$ is said to be \emph{stable  finite}, if there exists a stable bijection $\{1\leq m\leq n\}\to X$ for a unique $n\in L^0_s(\mathbb{N})$. 
If $X$ is  stable  finite, then there exist $(a_k)\in p(1)$, $(n_k)\subset \mathbb{N}$ and a countable family $(X_k)$ of subsets of $X$ with the cardinality of $X_k$ being equal to $n_k$ for all $k$ such that $X=\sum \text{st}(X_k)|a_k$.

\section{Stable $L^0$-modules and conditional vector spaces}\label{s:sec1}
 
In this section, we establish an equivalence of categories between the category of stable $L^0$-modules and the category of conditional real vector spaces. 
We show that a stable $L^0$-module is free if and only if it is the direct sum of finitely many copies of $L^0$. 
We construct a stable algebraic basis which can serve as a useful replacement in cases where we lack a classical basis.   
We use the upper case latin letters $E$ and $F$ to denote $L^0$-modules, and the boldface upper case latin letters $\mathbf{E}$ and $\mathbf{F}$ to denote conditional vector spaces. 

An $L^0$-module $E$ is an additive group on which the commutative ring $L^0$ acts by scalar multiplication. 
In particular, this implies that the associated measure algebra acts on $E$ by scalar multiplication with indicator functions which makes the following definition plausible. 

\begin{definition}
An $L^0$-module $E$ is said to be \emph{stable}, if for all $(x_k)\subset E$ and $(a_k)\in p(1)$ there exists a unique element $x\in E$ such that $1_{A_k}x=1_{A_k}x_k $ for all $k$. 
$x$ is called the concatenation of $(x_k)$ along $(A_k)$, and denoted formally by $x=\sum_k 1_{A_k} x_k$. 
\end{definition}

In \cite[Example 2.12]{kupper03} the existence of concatenations is not satisfied, and in \cite[Example 1.1]{zapataportfolio} their uniqueness fails. Thus an $L^0$-module is not a priori stable. 
However, we can always stabilize an $L^0$-module as follows. 

\begin{remark}
Let $E$ be an $L^0$-module. 
Let $E^s$ be the collection of all $(x_k, a_k)\subset E\times \mathcal{A}$ where $(a_k)\in p(1)$. 
Identify $(x_k,a_k)$ and $(y_k,b_k)$ whenever $1_{A_k\cap B_{h}} x_k =1_{A_k\cap B_h}y_{h}$ for all $h,k$. 
Defining $(x_k,a_k)+(y_k,b_k):=(x_k+y_h,a_k \wedge b_h)$ and $r(x_k,a_k)= (r x_k,a_k)$ provides $E^s$ with the structure of a stable $L^0$-module in which $E$ can be included via $x\mapsto (x,1)$. 
\end{remark}

\begin{definition}
A triplet $(\mathbf{E},+,\cdot)$ consisting of a conditional set and two conditional functions is said to be a \emph{conditional vector space}, if the carrier structure $(E,+,\cdot)$ is an $L^0$-module.\footnote{Suppose that $\textbf{E}$ is a conditional vector space. For $(x_k)\subset E$ and $(a_k)\in p(1)$ let $x:=\sum x_k|a_k$. Then the stability of the scalar product allows to prove that $1_{A_k}x=1_{A_k}x_k$ for each $k$. This means that $E$ is stable in the modular sense and $\sum 1_{A_k}x_k=\sum x_k|a_k$, that is, the abstract conditioning $\cdot|a$ is consistent with the modular conditioning $1_A\cdot$.}
\end{definition}
Now every stable $L^0$-module $E$ induces also a conditional vector space.  
Indeed, define on $E\times \mathcal{A}$ the equivalence relation $(x,a)\sim (y,b)$ whenever $a=b$ and $1_A x=1_B y$, and let $x|a$ denote the equivalence class of $(x,a)$. 
Then the collection $\mathbf{E}$ of all equivalence classes $x|a$ is a conditional set with carrier $E$. 
Since addition and scalar multiplication in $E$ are stable functions, $\mathbf{E}$ inherits the structure of a conditional vector space.

Recall that a function $f\colon E\to F$ is a module homomorphism, if it is $L^0$-linear. 
Observe that every module homomorphism satisfies $1_A f(1_A x) =1_A f(x) $ for all $a\in\mathcal{A}$ and $x\in E$. 
It follows that every module homomorphism is a stable function whenever $E$ and $F$ are stable modules. 
A conditional function $\mathbf{f}\colon \mathbf{E}\to\mathbf{F}$ is conditional linear, if the stable function $f\colon E\to F$ is a module homomorphism. 

Let $\textbf{CVect}_\mathbf{R}$ denote the category whose objects are conditional vector spaces and whose morphisms are conditional linear functions. 
Let $\textbf{Mod}^{\text{s}}_{L^0}$ be the category whose objects are stable $L^0$-modules and whose morphisms are module homomorphisms. 
The next equivalence readily follows from the previous considerations. 

\begin{theorem}\label{t:catequiv1}
The functor mapping a conditional vector space $\mathbf{E}$ to the $L^0$-module $E$ and a conditional linear function $\mathbf{f}$ to the module homomorphism $f$ is an equivalence of categories between $\textbf{CVect}_\mathbf{R}$  and $\textbf{Mod}^{\text{s}}_{L^0}$. 
\end{theorem}

Since $L^0$ is a commutative ring, every free $L^0$-module has a rank, see \cite[Theorem 2.1]{hungerford2003algebra}. 
\begin{proposition}\label{p:freemodule}
Suppose there exists an infinite partition in $p(1)$ and let $E$ be a free $L^0$-module. 
Then $E$ is stable if and only if $E$ has finite rank.
\end{proposition}

\begin{proof}
If $E$ has finite rank, then $E$ is finitely generated, say by $\{v_1,\ldots, v_n\}$. 
Let $(a_k)\in p(1)$ and $(x_k)\subset E$ . 
Each $x_k$ is of the form $r_1^k v_1+\ldots +r_n^k v_n$ for some $r_m^k\in L^0$, $m=1,\ldots, n$. 
Put $r_m:=\sum_k 1_{A_k} r_m^k$ for $m=1,\ldots,n$ and $x:=r_1 v_1+\ldots +r_n v_n$. 
Then $1_{A_k} x =1_{A_k} x_k $ for all $k$. 
Now, let $y\in E$ be satisfying $1_{A_k} y =1_{A_k} x_k $ for all $k$. 
If we prove that $x=y$, then we conclude that $E$ is stable. 
Indeed, $y$ is of the form $s_1 v_1+\ldots +s_n v_n$ for some $s_m\in L^0$, $m=1,\ldots, n$. 
Then, for each $k$, $0=1_{A_k}(x-y)=1_{A_k}(r_1-s_1)v_1+\ldots+1_{A_k}(r_n-s_n)v_n$. This means that $1_{A_k}(r_m-s_m)=0$ for all $k,m$, and thus $r_m=s_m$ for each $m$. We conclude that $x=y$.

Conversely, suppose that $E$ is stable. 
By contradiction, assume that $E$ has not finite rank. 
Since $E$ is free, $E$ has a basis $(x_i)_{i\in I}$ of determined length.  
Let $(a_k)\in p(1)$ be an infinite partition, $(i_k)$ be an injective infinite subfamily of $I$ and $x:=\sum_k 1_{A_k} x_{i_k}$. 
Then there exists a finite collection $J\subset I$ and $(r_i)_{i\in J}\subset L^0$ such that $x=\sum_{i\in J} r_i x_i$. 
Since $J$ is finite, there exists $i_k\in I\setminus J$ for some $k$. 
One has 
\[
0=1_{A_k} x  - 1_{A_k} x=\sum_{i\in J}  r_{i} 1_{A_k} x_{i}  - 1_{A_k} x_{i_k}. 
\]
Then necessarily $1_{A_k}=0$, but this contradicts $\mathbb{P}(A_k)>0$.   
\end{proof}
This motivates the following definition. 
\begin{definition}\label{d:basis}
Let $\mathbf{E}$ be a conditional vector space, 
$(x_{\m})_{\1\leq \m\leq \n}$ a  stable  finite family in $E$ and $(s_\m)_{\1\leq \m\leq \n}$ a  stable  finite family in $L^0$. 
A \emph{stable linear combination} is defined as\footnote{In the definition of a  stable  finite sum, we use the fact that every  stable  finite family can be written as the concatenation of the stable hulls of finite subfamilies along $(a_k)$.} 
\[
\sum_{\mathfrak{1}\leq \m\leq \n} r_\m x_{\m}:=\sum (\sum_{m=1}^{n_k} r_m x_m)|a_k, 
\]
where $\n=\sum_k 1_{A_k} n_k$.
The \emph{stable linear span} $\text{st-span}(S)$ of a non-empty subset $S\subset E$ is the stable hull of all stable linear combinations of the elements in $S$. 

A stable set $S\subset E$ is said to be a \emph{stable basis}, if the following conditions are satisfied: 
\begin{itemize}
\item[(i)] \emph{ stable  linear independence}: If $(x_{\m})_{\1\leq \m\leq \n}$  is a  stable  finite subset of $S$ and $(r_\m)_{\1\leq \m\leq \n}$ is a  stable  finite family in $L^0$, then 
\[
\sum_{\1\leq \m\leq \n} r_\m x_{\m}=0 \quad \text{implies} \quad r_\m=0  \quad \text{for all} \quad \1\leq \m\leq \n. 
\]
\item[(ii)] \emph{ stable  generating}: for every $x\in E$ there exist a  stable  finite family $(x_{\m})_{1\leq \m\leq \n}$ in $S$ and a stable finite family $(r_\m)_{\1\leq \m\leq \n}$ in $L^0$ such that 
\[
x=\sum_{\1\leq \m\leq \n} r_\m x_{\m}. 
\]
\end{itemize}
\end{definition}

\begin{theorem}\label{t:basis}
Every conditional vector space has a stable basis. 
\end{theorem}

\begin{proof}
Let $\mathbf{E}$ be a conditional vector space. 
Let $\mathscr{H}$ be the collection of all  stable  linearly independent sets in $E$. 
Then $\mathscr{H}$ is non-empty  since it contains every one-element family. 
Order $\mathscr{H}$ by inclusion. 
Let $(S_i)$ be a chain in $\mathscr{H}$. 
Then $\text{st}(\cup S_i)$ is a member of $\mathscr{H}$. 
By Zorn's lemma, $\mathscr{H}$ has a maximal element which we denote by $S_\ast$. 
By contradiction, suppose $\text{st-span}(S_\ast)\neq E$. 
In other words, 
\[
0<a_\ast:=\vee\{a\in\mathcal{A}\colon \text{ there is } x\in E \text{ such that } 1_Bx\not\in 1_B \text{st-span}(S_\ast) \text{ for all } 0<b\leq a\}. 
\]
By stability of $S_\ast$ and $E$, $a_\ast$ is attained by some $x_\ast\in E$ due to an exhaustion argument. 
Assume w.l.o.g.~that $a_\ast=1$. 
Then $S_{\ast\ast}:=\text{st-span}(S_\ast \cup\{x_\ast\})\in \mathscr{H}$ and $S_\ast \subsetneq S_{\ast\ast}$ which contradicts the maximality of $S_\ast$. 
\end{proof}

One can adapt Definition \ref{d:basis} to provide a definition of a stable basis for stable $L^0$-modules. 
It then follows readily from Proposition \ref{t:catequiv1}: 

\begin{corollary}
Every stable $L^0$-module has a stable basis. 
\end{corollary}

\begin{example}
Let $\n\in L^0_s(\mathbb{N})$ and $(L^0)^\n$ denote the set of all the stable functions $\{\1\leq \m\leq \n\} \to L^0$.  
Then $\{1_{\{\m=\n\}} \colon \1\leq \m\leq \n\}$ is a stable basis of $(L^0)^\n$. 
\end{example}

\begin{definition}
An $L^0$-module $E$ is said to be \emph{finitely generated} if it has a  stable  finite subset $S$ which  stably  generates $E$. 
\end{definition}

The following result is an extension of \cite[Theorem 2.8]{Cheridito2012}.

\begin{corollary}
Let $E$ be a  stable  finitely generated $L^0$-module. 
Then $E$ admits a  stable  finite basis. 
Moreover, there exist unique $(a_k)\in p(1)$ and $(n_k)\subset \mathbb{N}$ such that $1_{A_k}E$ is a free $1_{A_k}L^0$-module of rank $n_k$. 
\end{corollary}

In the setting of the previous corollary, we call $\n=\sum 1_{A_k} n_k$ the \emph{stable dimension} of $E$. 
As in the classical case, we have the following characterization. 

\begin{corollary}
If $E$ is  stable  finitely generated, then there exist $\n\in L^0_s(\mathbb{N})$ and a module isomorphism $E \to (L^0)^\n$. 
\end{corollary}

One can apply Proposition \ref{t:catequiv1} to derive the Hahn-Banach extension theorem\footnote{The Hahn-Banach extension theorem can be proved for modules without the stability assumption, see the discussion and the references in \cite{kupper03}.} for stable $L^0$-modules, see e.g.~\cite[Theorem 2.14]{kupper03}, from the conditional Hahn-Banach extension theorem \cite[Theorem 5.3]{DJKK13}. 
Recall that a function $f\colon E\to L^0$ is said to be \emph{$L^0$-convex}, if $f(r x + (1-r)y)\leq rf(x)+(1-r)f(y)$ for all $r\in L^0$ with $0\leq r \leq 1$ and every $x,y\in E$. 

\begin{theorem}\label{t:extension}
Let $E$ be a stable $L^0$-module. 
Let $p\colon E\rightarrow L^0$ be $L^0$-convex, $F\subset E$ a sub-module and  $f\colon F\rightarrow L^0$ $L^0$-linear and such that $f(x)\leq p(x)$ for all $x\in F$. 
Then there exists an $L^0$-linear function $\hat{f}\colon E\rightarrow L^0$ such that $\hat{f}|_F=f$ and $\hat{f}(x)\leq p(x)$ for all $x\in E$. 
\end{theorem}

\section{Topological $L^0$-modules and conditional topological vector spaces}\label{s:sec2}

We know of three types of module topologies on $L^0$-module. 
The first kind of topology is the $(\epsilon,\lambda)$-topolog which is introduced in \cite{haydon1991randomly} in the case of a single $L^0$-norm and for families of $L^0$-seminorms, see \cite{guo2009separation}. 
The second type of topology are locally $L^0$-convex topologies introduced in \cite{kupper03}. 
The last type of topology which we shall call stable topologies are intimately connected to conditional linear topologies introduced in \cite{DJKK13}. 
In this section, we focus on the relations among these three kind of topologies and their connection with conditional topological vector spaces which will be the foundation for the results in the following sections.  
Unless we mention otherwise, an $L^0$-module is stable. 
\begin{definition}
Let $E$ be an $L^0$-module. 
A function $p\colon E\to L^0_+$ is said to be an \emph{$L^0$-seminorm}, if 
\begin{itemize}
\item[(i)] $p(rx)=|r|p(x)$ for all $r\in L^0$, $x\in E$, 
\item[(ii)] $p(x+y)\leq p(x)+p(y)$ for all $x,y\in E$. 
\end{itemize}
In addition, if $p(x)=0$ implies $x=0$, then $p$ is said to be an \emph{$L^0$-norm}. 

Let $\mathbf{E}$ be a conditional vector space. 
A conditional function $\mathbf{p\colon E\to R_+}$ is said to be a \emph{conditional seminorm}, if the carrier function $p\colon E\to L^0_+$ is an $L^0$-seminorm.\footnote{Notice that (i) implies that $p$ is stable.} 
A conditional seminorm is a \emph{conditional norm}, if the carrier function is an $L^0$-norm. 
\end{definition}

We recall the definition of the $(\epsilon,\lambda)$-topology induced by a family of $L^0$-seminorms. 

\begin{definition}
Let $E$ be an $L^0$-module and $\mathscr{P}$ a non-empty collection of $L^0$-seminorms. 
For $x\in E$, let $\mathscr{U}_{\epsilon,\lambda}(x)$ denote the set of all 
\[
U_{x,N,\epsilon,\lambda}:=\{ y\in E \colon \mathbb{P}(\sup_{p\in N} p(x-y)<\epsilon )>1-\lambda\}, 
\]
where $\epsilon,\lambda\in\R_+$ with $\lambda<1$ and $N\subset \mathscr{P}$ is finite. 
The topology on $E$ generated by the base $\cup_{x\in X} \mathscr{U}_{\epsilon,\lambda}(x)$ is denoted by $\mathscr{T}_{\epsilon,\lambda}=\mathscr{T}_{\epsilon,\lambda}(\mathscr{P})$. 
\end{definition}

Notice that a  neighbourhood  $U_{x,N,\epsilon,\lambda}$ is not necessarily a stable set, nor is $\mathscr{U}_{\epsilon,\lambda}(x)$ a stable collection. 
The absolute value $|\cdot|\colon L^0\to L^0_+$ induces the topology of convergence in probability which is a linear as well as a module topology on $L^0$. 
More generally, given an $L^0$-module $E$ and any non-empty collection $\mathscr{P}$ of $L^0$-seminorms on $E$, the set $\mathscr{U}_{\epsilon,\lambda}(0)$ is a  neighbourhood  base of $0$ for a module topology where we assume that $L^0$ is endowed with the topology of convergence in probability. Notice that an $(\epsilon,\lambda)$-topology is always a linear topology as well, for more details see \cite[Proposition 2.6]{guo10}. 

Let $(p_k)$ be a countable family of $L^0$-seminorms  and $(a_k)\in p(1)$. 
Then $x\mapsto \sum_k 1_{A_k} p_k(x)$ is an $L^0$-seminorm. 
A collection $\mathscr{P}$ of $L^0$-seminorms is stable, if $\mathscr{P}\neq \emptyset$ and $\sum_k 1_{A_k} p_k\in \mathscr{P}$ for all $(p_k)\subset \mathscr{P}$ and every $(a_k)\in p(1)$.  
For a non-empty collection $\mathscr{P}$ of $L^0$-seminorms, we denote by $\text{st}(\mathscr{P})$ its stable hull and by $\text{st-sup}(\mathscr{P})$ the stable hull of the collection of all $\sup_{p\in N} p$, where $N\subset \mathscr{P}$ is finite. 
Next, we recall the definition of a locally $L^0$-convex topology induced by a family of $L^0$-seminorms. 

\begin{definition}
Let $E$ be an $L^0$-module and $\mathscr{P}$ a non-empty collection of $L^0$-seminorms. 
For $x\in E$, let $\mathscr{U}_0(x)$ be the collection of all 
\[
U_{x,N,r}:=\{y\in E\colon \sup_{p\in N} p(x-y)< r\}, 
\]
where $r\in L^0_{++}$ and $N \subset \mathscr{P}$ is finite. 
The topology on $E$ generated by the base $\cup_{x\in X}\mathscr{U}_0(x)$ is denoted by $\mathscr{T}_0=\mathscr{T}_0(\mathscr{P})$. 
\end{definition}

Notice that a  neighbourhood  $U_{x,N,r}$ is a stable set, however $\mathscr{U}_0(x)$ is not necessarily a stable collection. 
The condition $\mathscr{P}=\text{st-sup}(\mathscr{P})$ guarantees the stability of $\mathscr{U}_0(x)$, see \cite[Lemma 2.18]{kupper03} and the discussion following it. 
The locally $L^0$-convex topology induced by the absolute value $|\cdot|\colon L^0\to L^0_+$ is the interval topology w.r.t.~the order of a.e.~dominance. 
\begin{definition}
Let $E$ be an $L^0$-module and $\mathscr{P}$ a stable collection of $L^0$-seminorms. 
For $x\in E$, let $\mathscr{U}_s(x)$ denote the set of all 
\[
U_{x,\mathscr{N},r}:=\{y\in E\colon \sup_{p\in \mathscr{N}} p(x-y)< r\}, 
\]
where $r\in L^0_{++}$ and $\mathscr{N}\subset \mathscr{P}$ is  stable  finite. 
The topology generated by the base $\cup_{x\in E} \mathscr{U}_s(x)$ is denoted by $\mathscr{T}_s=\mathscr{T}_s(\mathscr{P})$.
\end{definition}

By construction, $\mathscr{U}_s(x)$ is a stable collection of stable sets.
Moreover it holds 
\[
\mathscr{U}_s(\sum_k 1_{A_k} x_k)=\sum_k 1_{A_k} \mathscr{U}_s(x_k) 
\]
for all $(x_k)\subset E$ and $(a_k)\in p(1)$. 
Since $\mathscr{B}:=\cup_{x\in E} \mathscr{U}_s(x)$ is a stable collection of stable sets and a topological base, it follows from  \cite[Proposition 3.5]{DJKK13} that the conditional collection $\mathcal{B}$ induced by $\mathscr{B}$ is the conditional topological base of a conditional topology on $\mathbf{E}$. 
The conditional topology induced by the absolute value on $L^0$ is isomorphic to the conditional Euclidean topology on $\mathbf{R}$, see \cite[Theorem 4.4]{DJKK13} and the dicussion preceding it. 
\begin{remark}
All three types of topologies are Hausdorff whenever $\sup_{p\in\mathscr{P}} p(x)=0$ if and only if $x=0$. 
In this case we say that $\mathscr{P}$ is \emph{separated}.  
We will assume the latter property throughout this article. 
\end{remark}
Next, we focus on the connection to conditional locally convex topological vector spaces.  
\begin{definition}
Let $E$ be an $L^0$-module. 
A subset $S\subset E$ is said to be \emph{$L^0$-convex}, if $rx+(1-r)y\in S$ for all $r\in L^0$ with $0\leq r\leq 1$ and every $x,y\in S$. 

Let $\mathbf{E}$ be conditional vector space. 
A conditional subset $\mathbf{S}$ of $\mathbf{E}$ is said to be \emph{conditional convex}, if its carrier is $L^0$-convex. 
\end{definition}
\begin{definition}
A topological $L^0$-module $(E,\mathscr{T})$ is said to be \emph{locally $L^0$-convex}, if there exists a  neighbourhood  base $\mathscr{U}$ of $0\in E$ consisting of stable $L^0$-convex sets in $E$. 
In this case, $\mathscr{T}$ is a \emph{locally $L^0$-convex topology}.
A conditional topological vector space $(\mathbf{E},\mathcal{T})$ is said to be \emph{conditional locally convex}, if there exists a conditional  neighbourhood  base of $\mathbf{0}$ in $\mathbf{E}$ consisting of conditional convex sets. 
\end{definition}
It is proved in \cite{kupper03,zapata2017characterization} that a topological $L^0$-module $(E,\mathscr{T})$ is locally $L^0$-convex if and only if $\mathscr{T}=\mathscr{T}_0(\mathscr{P})$ for some collection $\mathscr{P}$ of $L^0$-seminorms. 
The Hahn-Banach separation theorem \cite{kupper03} is proved under the additional assumption that the collection of $L^0$-seminorms satisfies $\mathscr{P}=\text{st-sup}(\mathscr{P})$. 
This amounts to the identity $\mathscr{U}_0(0)=\mathscr{U}_s(0)$, which makes $\mathscr{U}_0(0)$ a stable collection of stable subsets. 
In addition, the existence of a stable  neighbourhood  base of $0\in E$ consisting of stable subsets is  key to establish a categorical equivalence between locally $L^0$-convex modules and conditional locally convex topological vector spaces. 
This leads us to the following definition:
\begin{definition}
A locally $L^0$-convex module $(E,\mathscr{T})$ is called \emph{stable} if there exists a  neighbourhood  base $\mathscr{U}$ of $0\in E$, which is a stable collection of stable $L^0$-convex sets in $E$. 
In this case, $\mathscr{T}$ is a \emph{stable topology}.
\end{definition}
Suppose that $\mathscr{T}=\mathscr{T}_s(\mathscr{P})$ for some collection $\mathscr{P}$ of $L^0$-seminorms on $E$, then it is clear that $\mathscr{T}$ is a stable locally $L^0$-convex topology.  
Conversely, by an adaptation of a classical argument, it can be proven that for any stable locally $L^0$-convex topology $\mathscr{T}$  there is a stable family of $L^0$-seminorms $\mathscr{P}$ such that $\mathscr{T}=\mathscr{T}_s(\mathscr{P})$, see \cite{OZ2017stabil}.

Let  $\textbf{CLCS}$ denote the category whose objects are conditional locally convex topological vector spaces and whose morphisms are conditional continuous  linear functions. 
Let  $\textbf{LCS}^\text{s}_{L^0}$ denote the category whose objects are stable locally $L^0$-convex modules and whose morphisms are continuous $L^0$-linear functions. 
By \cite[Proposition 3.5]{DJKK13} and \cite[Proposition 3.11]{DJKK13}, we obtain the following equivalence of categories established in \cite[Theorem 1.2]{OZ2017stabil}.\footnote{The statement in \cite{OZ2017stabil} contains a mistake.  There $\mathscr{T}$ is identified with the set of carriers of conditional open sets. As stated here, $\mathscr{T}$ must be the topology generated by the latter set, which is a topological base for $\mathscr{T}$. The use of \cite[Proposition 3.11]{DJKK13} is not mentioned in the proof of \cite[Theorem 1.2]{OZ2017stabil}.}

\begin{theorem}\label{t:catequiv3}   
The functor mapping 
\begin{itemize}
\item a conditionally locally convex topological vector space $(\mathbf{E},\mathcal{T})$ to the stable locally $L^0$-convex module $(E,\mathscr{T})$, where $E$ is the carrier of $\mathbf{E}$ and $\mathscr{T}$ is the topology generated by the
set of carriers of conditional open sets.
\item and a conditional linear continuous function $\mathbf{f}\colon \mathbf{E}\to \mathbf{F}$ to the $L^0$-linear continuous function $f\colon E\to F$, 
\end{itemize}
is an equivalence of categories between $\textbf{CLCS}$ and $\textbf{LCS}^\text{s}_{L^0}$. 
\end{theorem}
Next, we will compare the different types of classical and conditional topologies. 
We need two preliminary results. 
\begin{proposition}\label{p:equivtop}
Let $\mathscr{P}$ be a non-empty collection of $L^0$-seminorms on an $L^0$-module $E$. 
Then 
$$\mathscr{T}_{\epsilon,\lambda}(\mathscr{P})=\mathscr{T}_{\epsilon,\lambda}(\text{st}(\mathscr{P}))=\mathscr{T}_{\epsilon,\lambda}(\text{st-sup}(\mathscr{P})).$$
\end{proposition} 

\begin{proof}
Since  $\mathscr{P}\subset\text{st}(\mathscr{P})\subset \text{st-sup}(\mathscr{P})$, it suffices to show that $\mathscr{T}_{\epsilon,\lambda}(\mathscr{P})=\mathscr{T}_{\epsilon,\lambda}(\text{st-sup}(\mathscr{P}))$. 
Since $\mathscr{P}\subset \text{st-sup}(\mathscr{P})$, then $\mathscr{T}_{\epsilon,\lambda}(\mathscr{P})\subset\mathscr{T}_{\epsilon,\lambda}(\text{st-sup}(\mathscr{P}))$. 

Conversely, let $(a_k)\in p(1)$, $(S_k)$ be a countable family of finite subcollections of $\mathscr{P}$ and $\epsilon,\lambda\in\R_+$ with $\lambda<1$. 
Put $q=\sum_k 1_{A_k} \sup_{p\in S_k} p$. 
Choose $m\in\N$ large enough such that $\mathbb{P}(\cup_{k>m} A_k)<\lambda/2$. 
Then we have 
 \begin{align*}
 U_{0,\{q\},\epsilon,\lambda}&=\{ x\in E \colon \mathbb{P}(\bigcup_{k\geq 1} \{\sup_{p\in S_k}(x)<\epsilon\}\cap A_k )>1-\lambda \}\\
 &\supset\{ x\in E \colon \mathbb{P}(\bigcup_{k=1}^m \{\sup_{p\in S_k}(x)<\epsilon\}\cap A_k )>1-\lambda \}\\
 &\supset\{x\in E \colon \mathbb{P}(\{\sup_{p\in \cup_{k=1}^m S_k}p(x) <\epsilon\}\cap\cup_{k=1}^m A_k)>1-\lambda  \}\\
 &\supset\{x\in E \colon \mathbb{P}(\sup_{p\in \cup_{k=1}^m S_k}p(x) <\epsilon)>1-\lambda + \mathbb{P}(\{\sup_{p\in \cup_{k=1}^m S_k}p(x)<\epsilon\}\cap\cup_{k>m} A_k)\}\\
 &\supset\{x\in E \colon \mathbb{P}(\sup_{p\in \cup_{k=1}^m S_k}p(x) <\epsilon)>1-\lambda + \mathbb{P}(\cup_{k>m} A_k)\}\\
 & =\{x\in E \colon \mathbb{P}(\sup_{p\in \cup_{k=1}^m S_k}p(x)<\epsilon)>1-\frac{\lambda}{2}\}, 
 \end{align*}
where the last set is a neighbourhood w.r.t.~$\mathscr{P}$.  
\end{proof}

\begin{proposition}\label{l:equivtop}
Let $\mathscr{P}_1,\mathscr{P}_2$ be collections of $L^0$-seminorms on an $L^0$-module $E$. If $\mathscr{T}_0(\mathscr{P}_1)=\mathscr{T}_0(\mathscr{P}_2)$, then $\mathscr{T}_{\epsilon,\lambda}(\mathscr{P}_1)=\mathscr{T}_{\epsilon,\lambda}(\mathscr{P}_2)$.
\end{proposition}

\begin{proof}
For $\epsilon,\lambda\in\R_+$ with $\lambda<1$ and finite $N\subset \mathscr{P}_1$,   
 choose finite $N^\prime\subset \mathscr{P}_2$ and $s\in L^0_{++}$ such that 
\[
\{x\in E\colon \sup_{p\in N^\prime} p(x)<s\}\subset \{x\in E\colon \sup_{p\in N} p(x)<\epsilon\}. 
\]
Let $(a_k)\in p(1)$ and $(\epsilon_k)\subset \R$ with $0<\sum_k 1_{A_k}\epsilon_k\leq s$. 
Choose $m\in\N$ large enough so that $\mathbb{P}(\cup_{k>m} A_k)< \lambda/2$. 
Then one has 
\begin{align*}
&\{ x\in E \colon\mathbb{P}(\sup_{p\in N} p(x) <\epsilon)>1-\lambda\}\\
&\supset \{ x\in E \colon \mathbb{P}(\cup_{k\in \N}\{\sup_{p\in N^\prime} p(x)< \epsilon_k\}\cap A_k)>1-\lambda \}\\
&\supset \{ x\in E \colon \mathbb{P}(\cup_{k=1}^m\{\sup_{p\in N^\prime} p(x)< \epsilon_k\}\cap A_k)>1-\lambda \}\\
&\supset \{ x\in E \colon \mathbb{P}(\{\sup_{p\in N^\prime} p(x)< \min_{1\leq k\leq m} \epsilon_k\}\cap \cup_{k=1}^m A_k)>1-\lambda \}\\
&\supset \{ x\in E \colon \mathbb{P}(\sup_{p\in N^\prime} p(x)< \min_{1\leq k\leq m} \epsilon_k)>1-\lambda+\mathbb{P}(\cup_{k>m} A_k) \}\\
&\supset \{ x\in E \colon \mathbb{P}(\sup_{p\in N^\prime} p(x)< \min_{1\leq k\leq m} \epsilon_k)>1-\frac{\lambda}{2} \}.
\end{align*}
This proves that $\mathscr{T}_{\epsilon,\lambda}(\mathscr{P}_1)\subset \mathscr{T}_{\epsilon,\lambda}(\mathscr{P}_2)$. 
By interchanging the roles of $\mathscr{P}_1$ and $\mathscr{P}_2$, it follows the claim.  
\end{proof}
In the following remark, we compare the three topologies. 
\begin{remark}\label{r:comparison}
Given a locally $L^0$-convex module $(E,\mathscr{T})$, there exists a collection $\mathscr{P}$ of $L^0$-seminorms such that $\mathscr{T}=\mathscr{T}_0(\mathscr{P})$, see \cite{kupper03,zapata2017characterization}. 
Proposition \ref{l:equivtop} tells us that the topology $\mathscr{T}_{\epsilon,\lambda}$ does not depend on $\mathscr{P}$. 
One can verify that $\mathscr{T}_s=\mathscr{T}_s(\text{st}(\mathscr{P}))=\mathscr{T}_0(\text{st-sup}(\mathscr{P}))$ is the coarsest stable topology such that $\mathscr{T}\subset\mathscr{T}_s$. Therefore, $\mathscr{T}_s$ does not depend on $\mathscr{P}$. Proposition \ref{p:equivtop} shows that $\mathscr{T}_s$ does not produce a new $(\epsilon,\lambda)$-topology.  
Consequently, $\mathscr{T}$ uniquely defines $\mathscr{T}_{\epsilon,\lambda}$ and $\mathscr{T}_s$ independent of the choice of the inducing family of $L^0$-seminorms.

For $S\subset E$ stable, we have $\text{cl}_{\epsilon,\lambda}(S)=\text{cl}_{0}(S)=\text{cl}_{s}(S)$ due to \cite[Theorem 3.12]{guo10}. 
Therefore, by \cite[Proposition 3.5]{DJKK13} and \cite[Proposition 3.7]{DJKK13}, the classical closure $\text{cl}_{0}(S)$ is the carrier of the conditional closure of $\mathbf{S}$ w.r.t.~the conditional topology $\mathcal{T}$. 

Denote by $E^\ast_{\epsilon,\lambda}$ the space of $L^0$-linear and $\mathscr{T}_{\epsilon,\lambda}$-continuous functions $f\colon E\to L^0$.   
Similarly, define the dual modules $E^\ast_{0}$ and $E^\ast_s$.  
By \cite[Proposition 2.14]{guo10},  we have $E^\ast_{\epsilon,\lambda}=E^\ast_{s}$. 
By \cite[Corollary 3.6]{guo2015random1}, if $\mathscr{T}$ if stable, then $E^\ast_{\epsilon,\lambda}=E^\ast_{0}$, and it follows from \cite[Proposition 3.11]{DJKK13} that $E^\ast_{\epsilon,\lambda}$ is the carrier of the conditional dual space $\mathbf{E^\ast}$ w.r.t.~the corresponding conditional topology. 
In this case, since we have the same dual module for the three topologies, we will denote them by $E^\ast=E^\ast[\mathscr{T}]$.
If $\mathscr{T}$ is not stable, then $E^\ast_0$ is not necessarily stable as the next example shows. 
However, by \cite[Theorem 3.6]{guo2015random1} and stabilization, we have $(E^\ast_0)^s=E^\ast_{\epsilon,\lambda}$, where we denote by $E^s$ the stabilization of $E$ as defined in Section \ref{s:sec1}. 
\end{remark}

\begin{example}
Let $E=(L^0)^\N$ be the set of all sequences in $L^0$. 
We consider the family of $L^0$-seminorms $\mathscr{P}:=\{p_n\colon n\in\N\}$, where $p_k((x_n)):=|x_k|$, $(x_n)\in E$.  
We claim that 
\begin{align*}
E^\ast_0&=\{(x_n)\in E \colon \exists m\in\N \textnormal{ such that }x_n=0\textnormal{ for all }n\geq m\},\\
E^\ast_{\epsilon,\lambda}&=E^\ast_s:=\{(x_n)\in E \colon \exists m \in L^0_s(\N) \textnormal{ such that }x_n=0\textnormal{ for all }n\geq m\}, 
\end{align*}
where $x_\n:=\sum_k 1_{A_k}x_{n_k}$ whenever $\n=\sum_k 1_{A_k}n_k$, $(a_k)\in p(1)$ and $(n_k)\subset\N$.
Notice that the latter space is the stable hull of the former. 
Observe also that these spaces are different whenever there exists an infinite $(a_k)\in p(1)$. 
By Remark \ref{r:comparison}, it suffices to prove the first equality. 
Any $x:=(x_n)$ with $x_n=0$ for all $n\geq m$ defines via
\[
f_x((y_n)):=\sum_{n \geq 1} x_n y_n 
\]
a $\mathscr{T}_0$-continuous $L^0$-linear function. 

Conversely, suppose that $f\in E^\ast_0$. 
By continuity, there exists $s\in L^0_{++}$ and $m\in \N$ such that $|f(x)|< 1$ whenever $x=(x_n)$ satisfies $|x_k|< s$ for all $k=1,\ldots,m$. 
Let $e_k\in E$ denote the sequence such that its $k^{\text{th}}$-coordinate is one and zero otherwise.  
Then $|f(n e_k)|=n|f(e_k)|\leq 1$ for every $k>m$ and $n\in \N$. 
This implies that $f(e_k)=0$ whenever $k>m$. 
Now define $x=(x_n)$ by $x_k:=f(e_k)$ for $k\leq m$ and $x_k:=0$ otherwise. 
Then we have $f_x=f$. 
\end{example}

Finally, we collect some examples of stable locally $L^0$-convex modules. 

\begin{examples}
\begin{enumerate}
\item Let $(X,\|\cdot\|)$ be a normed vector space, $L^0(X)$ the space of all strongly measurable functions modulo a.e.~equality, and define $\|\cdot \|\colon L^0(X)\to L^0$ by $\|x\|(\omega):=\|x(\omega)\|$ a.e. $\omega\in \Omega$. 
Then $L^0(X)$ is a stable $L^0$-module and $\|\cdot\|$ is an $L^0$-norm, see \cite[Section 5.1]{haydon1991randomly} and \cite[Section 4]{DJK16}. 
The space $L^0(X)$ can be constructed by a conditional completion, and thus is the carrier of a conditional Banach space, see \cite{DJK16}. 
\item We call $(E,F,\langle\cdot,\cdot\rangle)$ a \emph{stable dual pair}, if $E,F$ are stable $L^0$-modules and $\langle\cdot,\cdot\rangle\colon E\times F\to L^0$ is an $L^0$-bilinear form such that $\langle x,y\rangle=0$ for all $y\in F$ implies $x=0$, and $\langle x,y\rangle=0$ for all $x\in E$ implies $y=0$, see \cite{guo10} as well. 
From $L^0$-linearity follows stability, and therefore the associated conditional structure is a conditional dual pair in the sense of \cite[Definition 5.6]{DJKK13}. 
It can be checked that $\{|\langle \cdot, y\rangle|\colon y\in F\}$ is a stable collection of $L^0$-seminorms. 
\item Let $(X,Y,\langle \cdot,\cdot \rangle)$ be a dual pair of Banach spaces such that 
\begin{itemize}
\item[(a)] $|\langle x,y\rangle|\leq \|x\| \|y\|$ for all $x\in X$ and $y\in Y$,  
\item[(b)]  both norm-closed unit balls are weakly closed.
\end{itemize}
Then $\langle \cdot,\cdot \rangle$ extends to an $L^0$-duality pairing on $L^0(X)\times L^0(Y)$, see \cite[Lemma 3.1]{drapeau2017fenchel}. 
\item Let $(E,\mathscr{T})$ be a  locally $L^0$-convex module.  
Consider the canonical duality pairing between $E$ and $E^\ast_s$, and denote by $\sigma_{\epsilon,\lambda}(E,E^\ast_s)$, $\sigma_{0}(E,E^\ast_s)$ and $\sigma_{s}(E,E^\ast_s)$ the weak $(\epsilon,\lambda)$-topology, the weak $L^0$-topology and the weak stable topology on $E$ induced by the stable collection of $L^0$-seminorms $\{|\langle \cdot, x^\ast\rangle|\colon x^\ast\in E^\ast_s\}$, respectively. 
Similarly, we can introduce the three different types of weak$^\ast$-topologies $\sigma_{\epsilon,\lambda}(E^\ast_s,E)$, $\sigma_{0}(E^\ast_s,E)$ and $\sigma_{s}(E^\ast_s,E)$. 
\item Let $\mathcal{E}\subset \mathcal{F}$ be a sub-$\sigma$-algebra and $1\leq p\leq \infty$. 
Let $\bar{L}^0(\mathcal{E})$ denote the space of equivalence classes of $\mathbb{R}\cup\{+\infty\}$-valued random variables. 
Let $\Vert\cdot\Vert_p:L^0(\mathcal{F})\rightarrow\bar{L}^0(\mathcal{E})$ be defined by 
\[
\Vert x \Vert _{p}:=
\begin{cases}
\lim_{n\to \infty}\mathbb{E}[|x|^{p}\wedge n | \mathcal{E}]^{1/p}, & \text{ if } p<\infty,\\
\inf\{ y\in\bar{L}^{0}(\mathcal{E})\colon y\geq |x|\},  & \textnormal{ if } p=\infty.
\end{cases}
\]
Let $L^p_\mathcal{E}(\mathcal{F}):=\{x\in L^0(\mathcal{F})\colon \Vert x\Vert_p\in L^0(\mathcal{E})\}$. 
Then $(L^p_\mathcal{E}(\mathcal{F}),\Vert\cdot\Vert_p)$ is a stable $L^0(\mathcal{E})$-normed module for which one has the representation $L^p_\mathcal{E}(\mathcal{F})=L^0(\mathcal{E})L^p(\mathcal{F})$. 
For H\"older conjugates $(p,q)$ where $1\leq p < \infty$, the function $\langle\cdot,\cdot\rangle:L^p_\mathcal{E}(\mathcal{F})\times L^q_\mathcal{E}(\mathcal{F})\rightarrow L^0(\mathcal{E})$ defined by $(x,y)\mapsto\mathbb{E}[x y|\mathcal{E}]$ is an $L^0(\mathcal{E})$-duality pairing.\footnote{For $p=2$, the $L^2$-type module $L^p_\mathcal{E}(\mathcal{F})$ was introduced in \cite{HR1987conditioning}, see \cite{kupper04} for other values of $p$ and the representation result.} 
\end{enumerate}
\end{examples}
\section{Classical and conditional compactness}\label{s:compact}
In this section, we study to which extent basic compactness results in functional analysis can be established in topological $L^0$-modules. The classical notion of compactness does not bear much fruits. 
Indeed, it is proved in \cite{guo08} that the Banach-Alaoglu and the Banach-Bourbaki-Kakutani-\v{S}mulian theorem for the $(\epsilon,\lambda)$-topology are fulfilled if and only if the support of $E$ is essentially atomic. 
Since the corresponding $L^0$- and stable topology are finer than the $(\epsilon,\lambda)$-topology, there is also no hope to prove these theorems in full generality for stable and $L^0$-type topological modules.  
We will complement the previous finding by verifying that a stable locally $L^0$-convex module is anti-compact if and only if its support is atomless. 
This motivates the concept of stable compactness for stable locally $L^0$-convex modules which we will derive from conditional compactness. 
In the remainder of this article, we then proceed to show that all basic compactness results in functional analysis can be established in full generality for stable locally $L^0$-convex modules using stable compactness. 
In the finite dimensional case, we provide characterizations of stable compactness in terms of a.e.~convergence and compact set-valued mappings, respectively, which are relevant in applications. 

The support of an $L^0$-module $E$ is defined as
 \[
 \text{supp}(E):=\wedge\left\{ a\in\mathcal{A} \colon 1_{A^c}E=\{0\}\right\}.
 \]
 
\begin{lemma} 
\label{lem: clusterPoint}
Let $(r_n)\subset L^0_+$ be with $r_n\neq 0$ for all $n\in\N$. Suppose that $\{\sup_{n\in \N} \, r_n>0\}$ is atomless.  
Then there exists $r\in L^0_{++}$ such that $\mathbb{P}(r_n\geq r)>0$ for all $n\in\N$.
\end{lemma}

\begin{proof}
For each $n\in\N$, let $A_n:=\left\{r_n>0\right\}$. 
Since $r_n\neq 0$, one has that $\mathbb{P}(A_n)>0$ for all $n\in\N$. 
Let $B_1:=A_1$.  
Since $A_2\subset \{\sup_n r_n>0\}$ and the latter is atomless, we can choose $B_2\subset A_2$ with $0<\mathbb{P}(B_2)<\frac{1}{2}\mathbb{P}(B_1)$.
By induction, for each $n\geq 1$, we can find $B_{n+1}\subset A_{n+1}$ with $0<\mathbb{P}(B_{n+1})<\frac{1}{2}\mathbb{P}(B_n)$. 
For fixed $n\in\N$, it thus holds
\[
\mathbb{P}(B_k)<\frac{1}{2^{k-n}}\mathbb{P}(B_n)\quad\textnormal{ for all }k>n.
\]
Now, for each $n\in\N$, put $C_n:=B_n - \underset{k>n}\bigcup B_k$.  Then, 
\[
\mathbb{P}(C_n)\geq \mathbb{P}(B_n) -  \sum_{k>n}\mathbb{P}(B_k)> \mathbb{P}(B_n) -  \sum_{k>n}^\infty \frac{1}{2^{k-n}}\mathbb{P}(B_n)=0.
\]
Note that $(c_n)\in p(\vee c_n)$. 
Define $r:=\sum_{n\in\N} 1_{C_n}r_n/2  + 1_{\cap C_n^c}$. 
Then $r\in L^0_{++}$, and we conclude
\[
\mathbb{P}(r_n\geq r)\geq\mathbb{P}(C_n)>0\quad\textnormal{ for all }n\in\N.
\] 
\end{proof}

\begin{proposition}\label{p:anti-compact}
Let $(E,\mathscr{T})$ be a stable locally $L^0$-convex module with $\text{supp}(E)>0$. 
Then $E$ is anti-compact if and only if $\text{supp}(E)$ is atomless.
\end{proposition}

\begin{proof}
Let $\mathscr{P}$ be a stable collection  of $L^0$-seminorms which induce $\mathscr{T}_s$. 
For the sake of contradiction, suppose that $K\subset E$ is an infinite compact set. 
Then there exists an injective sequence $(x_n)$ in $K$ which has a cluster point $x_0\in K$. 
We assume w.l.o.g.~that $x_0\neq x_n$ for all $n\in\N$. 
Since $\mathscr{P}$ is separated, for every $n\in\N$ there is $p_n\in \mathscr{P}$ such that $p_n(x_0-x_n)\neq 0$. 

Now, define $\rho:E\times E\rightarrow L^0$ by 
\[
\rho(x,y):=\sum_{n\in\N}\frac{1}{2^n}\frac{p_n(x-y)}{1+p_n(x-y)}.
\] 
Then $\rho(x_n,x_0)\neq 0$ for all $n\in\N$ and $\{\sup_n\rho(x_n,x_0)>0\}\subset\text{supp}(E)$. 
By Lemma \ref{lem: clusterPoint}, there is $r\in L^0_{++}$ such that $\mathbb{P}(\rho(x_n,x_0)\geq r)>0$ for all $n\in\N$.
Let 
\[
V_r(x_0):=\left\{ y\in E\colon \rho(x_0,y)< r \right\}, 
\]
and notice that $x_n\notin V_r(x_0)$ for all $n\in\N$. 
If we would show that $V_r(x_0)$ is a  neighbourhood  of $x_0$, then we obtain the desired contradiction. 
First, observe that there exist $(a_k)\in  p(1)$ and $(n_k)\subset\N$ such that
\[
\sum_{k\in\N} 1_{A_k}\sum_{n> n_k} \frac{1}{2^n}\frac{p_n(x_0-y)}{1+ p_n(x_0-y)}< \frac{r}{2}\quad\textnormal{ for all }y\in E.  
\]
Second, let $s\in L^0_{++}$ be small enough such that if $1_{A_k}p_n(x_0-y)\leq 1_{A_k}s$ for $n<n_k$, then 
\[
\sum_{k\in\N} 1_{A_k}\sum_{1\leq n\leq n_k} \frac{1}{2^n}\frac{p_n(x_0-y)}{1+p_n(x_0-y)}< \frac{r}{2}.
\]
Put $n:=\sum_k 1_{A_k}n_k$ and $\mathscr{N}:=\{p_\m\colon \1\leq \m\leq\n\}\subset \mathscr{P}$. 
Then $U_{x_0,\mathscr{N},s}\subset V_r(x_0)$.

As for the converse, suppose that there exists an atom $A\subset\text{supp}(E)$. 
Notice that from stability of an $L^0$-seminorm $p$ it follows that $1_Bp(1_B x)=p(1_B x)$ for all $x\in E$ and $B\in \mathcal{F}$. 
Thus, since $A$ is an atom, for all $p\in\mathscr{P}$ and $x\in E$ there exists $\eta_{x,p}\in \R_+$ such that $p(1_A x)=1_A \eta_{x,p}$. 
Let $\hat{p}:E\rightarrow \R$ be defined by $\hat{p}(x):=\mathbb{E}[p(1_A x)]$ for each $p\in \mathscr{P}$. 
Since $\mathscr{P}$ is separated, there exists $x_0\in E$ and $p_0\in \mathscr{P}$ such that $\hat{p}_0(x_0)>0$. 
Then $\hat{\mathscr{P}}:=\{\hat{p}\colon p\in \mathscr{P}\}$ is a set of real-valued seminorms which induces on the $1$-dimensional vector space $E_0:=\{1_A  r x_0\colon r\in L^0\}$ a  locally convex Hausdorff topology. 
By the Heine-Borel theorem, the unit ball $K$ of $E_0$ is infinite compact. 
Since $\mathscr{P}$ and $\hat{\mathscr{P}}$ induce the same topology on $E_0$, $K$ is also compact in $E$. 
\end{proof}

\begin{remark}
Next, we introduce the notions of stable topology, stable filter, stable compactness and stable metric on a stable subset of an $L^0$-module. 
These definitions make also sense on stable sets of functions as defined in \cite{jamneshan2017measures}, or more generally on the carrier of a conditional set. 
A stable structure can be viewed as an interpretation of a corresponding conditional structure in classical set theory. The definitions of a conditional filter, conditional ultrafilter, conditional compactness and conditional metric are introduced in \cite{DJKK13}, where the relation between classical and conditional topology is studied systematically, in which the following definitions of stable structures are  implicitly applied.  
\end{remark}

\begin{definition}
A topology on a stable subset $S$ of an $L^0$-module is said to be \emph{stable} if it admits a topological base which is a stable collection of stable sets.\footnote{It can straightforwardly be checked that if $\mathscr{T}$ is a stable topology, then $\mathscr{T}$ is a stable collection of not necessarily stable sets. Notice that, in general, the union of stable sets in not stable.} 
\end{definition}

\begin{definition}
A filter on a stable subset $S$ of an $L^0$-module is said to be \emph{stable} if it admits a filter base which is a stable collection of stable sets. 
 A stable filter is called a \emph{stable ultrafilter} if it is a maximal element in the set of all stable filters on $S$. 
 \end{definition}
 
\begin{definition}
Let $\mathscr{T}$ be a stable topology on $S$. 
Then $S$ is said to be \emph{ stable  compact} if every stable filter $\mathscr{F}$ on $S$ has a cluster point in $S$. 
\end{definition}

Similarly to classical compactness, we can prove the following characterizations of stable compactness. 
 
\begin{proposition}\label{thm: StCompactness}
Let $\mathscr{T}$ be a stable topology on $S$. Then the following are equivalent. 
\begin{itemize}
\item[(i)] $S$ is  stable  compact. 
\item[(ii)] Every stable ultrafilter on $S$ has a cluster point. 
\item[(iii)] For every stable collection $\mathscr{O}$ of stable and open sets with $S=\cup \mathscr{O}$ there exists a  stable  finite subcollection $\tilde{\mathscr{O}}\subset\mathscr{O}$ such that $S=\cup \tilde{\mathscr{O}}$. 
\item[(iv)] Every stable collection $\mathscr{C}$ of stable and closed sets in $S$ with $\cap \tilde{\mathscr{C}}\neq\emptyset$ for every  stable  finite subcollection $\tilde{\mathscr{C}}\subset\mathscr{C}$ satisfies $\cap \mathscr{C}\neq\emptyset$.
\end{itemize}
\end{proposition}

\begin{proof}
The claim can be established similarly to \cite[Proposition 3.25]{DJKK13} by applying \cite[Proposition 3.25]{DJKK13} and a straightforward adaptation of  \cite[Theorem 3.16]{DJKK13} to the present context. 
\end{proof}

To be able to prove a Tychonoff's theorem, one needs a new notion of product topology since the product of a family of stable topologies is not necessarily a stable topology, see \cite[Example 1.1]{OZ2017stabil}. 
This motivates the following construction.\\

\begin{example}
Let $(S_i,\mathscr{T}_i)$ be a non-empty family of stable topological spaces with base $\mathscr{B}_i$ for each $i$. 
The collection $\mathscr{B}$ of  all $\sum_{k}1_{A_k} \prod U_i^k$, where $U_i^k=S_i$ for all but finitely many $U^k_i\in \mathscr{B}_i$ and $(a_k)\in p(1)$, is a stable collection of stable sets and the base for a stable topology on $\prod S_i$, referred to as the \emph{stable product topology}. 
If $\mathcal{B}$ is the conditional collection of conditional subsets of $\prod_i\textbf{S}_i$ induced by stable sets in $\mathscr{B}$, then the conditional topology generated by $\mathcal{B}$ is the conditional product topology on  $\prod_i\textbf{S}_i$, for a definition see \cite[Examples 3.9(2)]{DJKK13}. 
\end{example}

The conditional Tychonoff's theorem \cite[Theorem 3.28]{DJKK13} yields the following 

\begin{theorem}
Let $(S_i,\mathscr{T}_i)_{i\in I}$ be a non-empty family of stable topologies. 
Then the Cartesian product $\prod_i S_i$ is  stable  compact w.r.t.~the stable product topology if and only if $S_i$ is  stable  compact for all $i\in I$. 
\end{theorem}

Next, we study compactness in stable metric spaces. 

\begin{definition}
Let $E$ be an $L^0$-module and $S\subset E$ stable. 
A \emph{stable metric} is a stable function $d:S\times S\rightarrow L^0_+$ such that for all $x,y,z\in S$ it holds
\begin{itemize}
\item[(i)]$d(x,y)=0$ if and only if $x=y$, 
\item[(ii)] $d(x,y)=d(y,x)$, 
\item[(iii)] $d(x,z)\leq d(x,y)+d(y,z)$. 
\end{itemize}
\end{definition}

The collection of all $B_r(x):=\{y\in S\colon d(x,y)<r\}$, where $r\in L^0_{++}$ and $x\in S$, forms a stable collection of stable sets and the base for a stable topology on $S$. 
We introduce the following definitions. 
\begin{itemize}
\item A stable sequence  $(x_\n)$ in stable metric space is said to be \emph{Cauchy} if for every $r\in L^0_{++}$ there exists $\n_0\in L^0_s(\N)$ such that $d(x_\n,x_\m)\leq r$ for all $\n,\m\geq \n_0$. 
\item A stable metric space is said to be  
\begin{itemize}
\item \emph{stable  complete} if every Cauchy stable sequence is convergent;\footnote{Here we understand convergence as convergence of $(x_\n)$ as a net.}
\item \emph{stable  sequentially compact} if for every stable sequence there exists a convergent stable subsequence;\footnote{A stable subsequence $(y_\n)$ is said to be a stable subsequence of $(x_\n)$ if there exists an increasing stable sequence $(\n_k)$ with $y_{k}=x_{\n_k}$ for all $k\in L^0_s(\mathbb{N})$.} 
\item  \emph{bounded} if $\sup\{d(x_0,x)\colon x\in S\}<\infty$ for some $x_0\in S$;
\item  \emph{stable  totally bounded} if for every $r\in L^0_{++}$ there exists a  stable  finite subset $\tilde{S}\subset S$ such that $S=\cup_{x\in\tilde{S}} B_r(x)$;  
\item  \emph{stable  separable} if there exists a stable sequence $(x_\n)$ in $S$ such that for every $x\in S$ there exists a stable subsequence $(x_{\m})$ of $(x_\n)$ such that $x=\lim x_{\m}$. 
\end{itemize}
\end{itemize}

We have the Heine-Borel theorem for stable metric spaces whose proof can be carried out similarly to the proof of \cite[Theorem 4.6]{DJKK13}. 

\begin{theorem}\label{thm: GenHeineB}
Let $S$ be a stable metric space. 
Then the following are equivalent. 
\begin{itemize}
\item[(i)] $S$ is  stable  compact. 
\item[(ii)] If $(R_\n)$ is a decreasing stable family of stable and closed sets in $S$, then $\cap R_\n\neq\emptyset$. 
\item[(iii)] $S$ is  stable  sequentially compact. 
\item[(iv)]  $S$ is  stable  totally bounded and  stable  complete. 
\end{itemize}
\end{theorem}

If the underlying $L^0$-module $E$ is  stable  finitely generated, then we have:

\begin{corollary}\label{thm: HeineBorel}
Let $(E,\Vert\cdot\Vert)$ be a  stable  finitely generated $L^0$-normed module. 
Then the stable unit ball $\{ x\in E \colon\Vert x\Vert\leq 1 \}$ is  stable  compact.
\end{corollary}

An example of a finitely generated $L^0$-normed module is $(L^0)^d$, $d\in\N$, which is endowed with the Euclidean $L^0$-norm $\Vert x\Vert=( \sum_{i=1}^d x_i^2)^{1/2}$.  
A subset $S\subset (L^0)^d$ is said to be \emph{sequentially closed} if $x\in S$ whenever $(x_n)\subset S$ and $\lim_n x_n=x$ a.e.   
In this case, we have the following characterizations of stable compactness. 

\begin{theorem}
Let $K$ be a subset of $(L^0)^d$. 
Then the following are equivalent. 
\begin{itemize}
\item[(i)] $K$ is  stable  compact. 
\item[(ii)] $K$ is stable, closed and  bounded. 
\item[(iii)] $K$ is stable, sequentially closed and bounded.\footnote{In the context of the conditional analysis on $\R^d$, this equivalent characterization of stable compactness  is applied to prove that a stable sequentially lower semi-continuous function attains its minimum on a stable, sequentially closed and  stable  bounded set, see \cite[Theorem 4.4]{Cheridito2012}. This result was subsequently applied in e.g.~\cite{kupper08,JKZ2017control} to solve stochastic control problems.}
\item[(iv)] $K$ is the set of measurable selectors of a non-empty measurable compact-valued mapping $S_K:\Omega\rightrightarrows \R^d$.\footnote{In \cite[Theorem 4.2]{JKZ2017control}, a characterization of measurable closed-valued mappings $S:\Omega\rightrightarrows \R^d$ in terms of stable sequentially closed subsets of $(L^0)^d$ is proved, which according to the analysis in this article yields a one-to-one correspondence between conditional closed subsets of a conditional Euclidean space and measurable closed-valued maps in $\mathbb{R}^d$.}  
\end{itemize}
In this case, the mapping $S_K:\Omega\rightrightarrows \R^d$ is unique modulo a.e.~equality.  
\end{theorem}

\begin{proof}
$(i)\Leftrightarrow(ii)$: Theorem \ref{thm: HeineBorel}.

$(ii)\Leftrightarrow(iii)$: Follows from the fact that a stable set $K\subset (L^0)^d$ is sequentially closed if and only if it is closed in the stable topology.   

$(iv)\Rightarrow (i)$: Let $S:\Omega\rightrightarrows \R^d$ be a measurable compact-valued mapping with $\text{dom}(S)=\Omega$. 
Then $K=\{ x\in (L^0)^d \colon x(\omega)\in S(\omega) \text{ for all }\omega\in \Omega\}$ is non-empty due to the Kuratowski-Ryll-Nardzewski theorem (see e.g.~\cite[Theorem 18.13]{aliprantis01}). 
By inspection, $K$ is stable. 
Since the Euclidean norm $\Vert\cdot\Vert:\R^d\rightarrow\R$ is continuous, by the Measurable Maximum Theorem (see e.g.~\cite[Theorem 18.19]{aliprantis01}), the function $$r:\Omega\rightarrow \R,\quad   r(\omega):=\underset{x\in S(\omega)}\max \Vert x\Vert$$ is an element of $L^0$. 
This shows that $K\subset B_r(0)$ which means boundedness. 
To show that $K$ is closed, let $x\in\text{cl}(K)$. 
There is $x_n\in K$ with $\Vert x-x_n\Vert\leq\frac{1}{n}$ for all $n\in\N$.  
 Let $\tilde{x}_n$ be a representative of $x_n$ each $n\in\N$. 
 Then $x$ has a representative $\tilde{x}$ such that $\Vert\tilde{x}(\omega)-\tilde{x}_n(\omega)\Vert\leq\frac{1}{n}$ for all $\omega\in\Omega$. 
 Therefore, $\tilde{x}(\omega)\in \text{cl}(S(\omega))=S(\omega)$ for all $\omega\in \Omega$ which implies $x\in K$. 
 
$(i)\Rightarrow(iv)$: Suppose that $K\subset (L^0)^d$ is  stable  compact. 
Enumerate $\mathbb{Q}^d:=\{q_1,q_2,\ldots\}$. 
Notice that $\mathbb{Q}^d$ can be considered as a subset of $L^0_s(\mathbb{Q}^d)$ via $q\mapsto 1_\Omega q$. 
 For each $n\in\N$, let $r_n:=\inf\{\Vert q_n-x\Vert\colon x\in K\}$. 
 For each $m\in\N$, since $K$ is stable, we find $x(n,m)\in K$ with $\Vert q_n- x(n,m)\Vert\leq r_n + \frac{1}{3 m}$. 
 For $\n,\m\in L^0_s(\N)$ of the form $\n=\sum_k 1_{A_k}n_k$ and $\m=\sum_h 1_{B_h}m_h$, define $r_\n:=\sum_k 1_{A_k}r_{n_k}$, $q_\n:=\sum_k 1_{A_k}q_{n_k}$ and $x(n,m):=\sum_{k,h} 1_{A_k\cap B_h}x_{n_k,m_h}$. 
 Then, 
 \begin{equation}
 \label{eq: xnm}
 \Vert q_\n - x(\n,\m)\Vert \leq r_\n + \frac{1}{3\m}\quad\textnormal{ for all }\n,\m\in L^0_s(\N).
\end{equation}   
Put $\mathscr{H}:=\{x(\n,\m)\colon \n,\m\in L^0_s(\N)\}$. 
 By stability, $\mathscr{H}\subset K$. 
 We claim that $\mathscr{H}$ is dense in $K$. 
 Indeed,  for any $x\in K$ and $\m\in L^0_s(\N)$, since $L^0_s(\mathbb{Q}^d)$ is dense in $(L^0)^d$, there exists $q_\n\in L^0_s(\Q^d)$ such that $\Vert q_\n- x\Vert\leq 1/3 \m$.
 Moreover, note that $r_{\n}\leq \Vert q_\n- x\Vert \leq 1/3 \m$. 
The previous two inequalities together with (\ref{eq: xnm}) yield 
$$\Vert x - x(\n,\m)\Vert\leq \Vert x - q_\n\Vert + \Vert q_\n - x(\n,\m)\Vert\leq \frac{1}{3\m} + \frac{1}{3\m} +\frac{1}{3\m}\leq \frac{1}{\m}.$$ 
Now enumerate $\N\times\N:=\{(n_h,m_h)\colon h\in\N\}$. 
 For each $h\in L^0_s(\N)$ of the form $h=\sum_k 1_{A_k}h_k$, let $y_{h}:=\sum_k 1_{A_k}x(n_{h_k},m_{h_k})$. 
Considering $\N$ as a subset of $L^0_s(\N)$ via $n\mapsto 1_{\Omega}n$, we thus obtain a sequence $(y_n)$. 
Since $K$ is  stable  bounded, we can choose a representative $\tilde{y}_n$ of $y_n$ such that $\{\tilde{y}_n(\omega)\colon n\in \N\}$ is bounded for all $\omega\in\Omega$.  
We define 
\[
S_K(\omega):=\text{cl}(\{ \tilde{y}_n(\omega) \colon n\in\mathbb{N}\}), \quad \omega\in\Omega. 
\]
By Castaing's representation theorem (see e.g.~\cite[Corollarly 18.14]{aliprantis01}), $S_K$ is a measurable. 
Moreover, $S_K$ is non-empty compact-valued. 
By inspection, $K$ is the set of equivalence classes of all measurable selectors of $S_K$. 

Finally, let us show that if $K$ is the set of equivalence classes of all measurable selectors of two measurable non-empty compact-valued mappings $S$ and $S^\prime$, then $S(\omega)=S^\prime(\omega)$ a.e. Indeed, since $S$ and $S^\prime$ are measurable, they are graph-measurable due to \cite[Theorem 18.6]{aliprantis01}. 
Thus $A_x:=\left\{\omega\in\Omega \colon x\in S(\omega)\right\}$ and $B_x:=\left\{\omega\in\Omega \colon x\in S^\prime(\omega)\right\}$ are measurable for all $x\in \R^d$. 
It suffices to prove that $\mathbb{P}(A_x\Delta B_x)=0$ for all $x\in \R^d$, where $\Delta$ denotes symmetric difference. 
 By contradiction, suppose that for instance $\mathbb{P}(A_x\setminus B_x)>0$ for some $x\in\R^d$. 
Let $C:=A_x\setminus B_x$ and choose $x_0\in K$. 
Since $K$ is stable, $z:=1_C x + 1_{C^c}x_0\in K$. 
Hence $z$ is a measurable selector of $K$,  and thus $\tilde{z}(\omega)\in S^\prime(\omega)$ for all $\omega\in \Omega$ for a representative $\tilde{z}$ of $z$, which is a contradiction since $\mathbb{P}(C)>0$.  
\end{proof}

Let $(S,\mathscr{T})$ be a stable topological space. 
A function $f:S\rightarrow\bar{L}^0$ is said to be lower semi-continuous if $\{x\in S\colon f(x)\leq\eta\}$ is closed for all $\eta\in \bar{L}^0$. 
The following result is an extension of \cite[Theorem 4.4]{Cheridito2012} to the infinite dimensional case. 

\begin{theorem}\label{thm: lowSemCompact}
Let $(K,\mathscr{T})$ be a stable topological space and $f:K\rightarrow L^0$ stable and lower semi-continuous. 
If $K$ is  stable  compact, then there exists $x_0\in K$ such that $f(x_0)=\min_{x\in K} f(x)$.  
\end{theorem}

\begin{proof}
Define $r:=\inf_{y\in K} f(y)$ and 
\[
C_\n:=\left\{ x\in K\colon f(x) \leq r+\frac{\1}{\n}\right\}, \quad \n\in L^0_s(\N). 
\]
By inspection, $\{C_\n \colon \n\in L^0_s(\N)\}$ is a stable collection of stable and closed subsets of $K$. 
In fact, since $C_\n\cap C_\m=C_{\n\wedge \m}$, the collection $\{C_\n\colon \n\in L^0_s(\N)\}$ is a stable filter base on $K$. 
Denoting by $\mathcal{F}$ the induced stable filter, by stable compactness, $\mathcal{F}$ has a cluster point $x_0$. 
In particular, $x_0\in \text{cl}(C_\n)=C_\n$ for all $\n\in L^0_s(\N)$, and thus $f(x_0)=r$.   
\end{proof}

\section{Functional analysis in $L^0$-modules}\label{s:funcana}

In the remainder of this article, we will establish module variants of results in functional analysis. 
In doing so, the comparison results in Section \ref{s:sec2} are key which we recall in the following. 
Given a locally $L^0$-convex module $(E,\mathscr{T})$, according to the discussion in Remark \ref{r:comparison}, it uniquely defines two topologies $\mathscr{T}_{\epsilon,\lambda}$ and $\mathscr{T}_s$ with $\mathscr{T}_{\epsilon,\lambda}\subset \mathscr{T}\subset \mathscr{T}_s$.
Moreover, for a stable subset $S\subset E$, we have $\text{cl}_{\epsilon,\lambda}(S)=\text{cl}_0(S)=\text{cl}_s(S)$, and therefore we just write $\text{cl}(S)$. 
The stable topology $\mathscr{T}_s$ is uniquely related with a conditional locally convex topology $\mathcal{T}$ on $\mathbf{E}$, and $\text{cl}(S)$ is the carrier of the conditional closure of $\mathbf{S}$. 
Further, recall that if $\mathscr{T}$ is stable collection, that is, if $\mathscr{T}=\mathscr{T}_s$, then it holds the identity $E^\ast_{\epsilon,\lambda}=E^\ast_0=E^\ast_s$ for the topological $L^0$-dual spaces. 
In this case, we just write $E^\ast$.  
If $\mathscr{T}$ is non-stable, then we only have  $E^\ast_0\subset E^\ast_{\epsilon,\lambda}=E^\ast_s$ since the first space may be non-stable. 
We will make frequent use of these results without further comment. 

Throughout, we will suppose that any $L^0$-module considered $E$ has support $1$. 
For $x\in E$, we define $\text{supp}(x):=\wedge\{a\in\mathcal{A}\colon 1_{A^c}x=0\}$.
We start with the hyperplane separation theorem. 
For a discussion of variants of this theorem w.r.t.~the $L^0$- and $(\epsilon,\lambda)$-topology, we refer to the comparison in \cite{guo10}. 
We state the locally $L^0$-convex version obtained in \cite{kupper03}. 
\begin{theorem}
Let $(E,\mathscr{T})$ be a locally $L^0$-convex module and $S_1,S_2\subset E$ non-empty and $L^0$-convex with $S_1$ open. 
If $1_A S_1\cap 1_A S_2=\emptyset$ for all $a\in\mathcal{A}$ with $a>0$, then there exists $f\in E^\ast_0$ such that 
$$f(y)\leq f(z)\quad\text{ for all }y\in S_1, z\in S_2.$$   	
\end{theorem} 
With the notion of stable compactness, we can strengthen the previous statement and obtain the following strong separation result which has not been available before. 
\begin{theorem}
Let $(E,\mathscr{T})$ be a stable locally $L^0$-convex module. 
Let $S_1,S_2\subset E$ be stable $L^0$-convex with $S_1$  stable  compact and $S_2$ closed. 
If  $1_A S_1\cap 1_A S_2=\emptyset$ for all $a\in\mathcal{A}$ with $a>0$, then there exists $f\in E^\ast$ and $r\in L^0_{++}$ such that 
$$f(x)+r< f(y)\quad\text{ for all }x\in S_1,y\in S_2.$$
\end{theorem} 
\begin{proof}
The claim follows from \cite[Theorem 5.5(ii)]{DJKK13}. 
\end{proof}
A Fenchel-Moreau theorem for stable locally $L^0$-convex modules is proved in \cite[Theorem 3.8]{kupper03}. 
It was shown in \cite{guo2015random1} that this result is valid for the $L^0$-topology (see \cite[Theorem 5.5]{guo2015random1}) and the $(\epsilon,\lambda)$-topology (see \cite[Theorem 5.3]{guo2015random1}). 
Recall that a function $f:E\rightarrow \bar{L}^0$ is said to be \emph{proper} if $f(x)>-\infty$ for all $x\in E$ and $f(x)\in L^0$ for some $x\in E$. 
Let $(E,\mathscr{T}_0)$ be a  locally $L^0$-convex module and $f:E\rightarrow\bar{L}^0$ a proper function. 
We can consider the following three module conjugates of $f$: 
 \begin{align*}
f^\ast_{\epsilon,\lambda}&:E^\ast_{\epsilon,\lambda}\rightarrow\bar{L}^0, \;f^\ast_{\epsilon,\lambda}(g):=\sup_{x\in E} (g(x)-f(x)), \\
f^\ast_0&:E^\ast_{0}\rightarrow\bar{L}^0,  \;f^\ast_0(g):=\sup_{x\in E} (g(x)-f(x)), \\
f^\ast_s&:E^\ast_s\rightarrow\bar{L}^0,  \;f^\ast_s(g):=\sup_{x\in E} (g(x)-f(x)). 
\end{align*}
Since $E^\ast_{\epsilon,\lambda}=E^\ast_s$, we have $f^\ast_{\epsilon,\lambda}=f^\ast_s$. 
Similarly, we can define three module bi-conjugates of $f$. 
\begin{theorem}
Let $(E,\mathscr{T}_0)$ be a locally $L^0$-convex module. 
Let $f:E\rightarrow \bar{L}^0$ be proper, $L^0$-convex and lower semi-continuous. 
Then 
$$f(x)=f_0^{\ast\ast}(x)=f_s^{\ast\ast}(x)=f_{\epsilon,\lambda}^{\ast\ast}(x),\quad x\in E.$$
\end{theorem}
\begin{proof}
Since $f$ is stable, if it is lower semi-continuous w.r.t.~$\mathscr{T}_0$, then it is also lower semi-continuous w.r.t.~$\mathscr{T}_s$ and $\mathscr{T}_{\epsilon,\lambda}$. 
By \cite[Theorem 3.8]{kupper03}, we have $f^{\ast\ast}_s=f$. 
Since $E_s^\ast=E_{\epsilon,\lambda}^\ast$, one has $f^{\ast\ast}_{\epsilon,\lambda}=f^{\ast\ast}_s=f$. 
Further it holds $(E_0^\ast)^s=E_s^\ast$, and thus, for fixed $x\in E$, one has 
\begin{align*}
f^{\ast\ast}_s(x)&=\sup_{g\in E^\ast_s} (g(x)-f^\ast_s(g)) =\sup_{g\in (E^\ast_0)^s} (g(x)-f^\ast_s(g))\\
&=\sup_{g\in E^\ast_0} (g(x)-f^\ast_s(x))=\sup_{g\in E^\ast_0} (g(x)-f^\ast_0(x))=f^{\ast\ast}_0(g).
\end{align*}
\end{proof}
In \cite[Corollary 3.4]{guo2009separation}, a module variant of Mazur's lemma for the $(\epsilon,\lambda)$-topology is proved. 
The conditional version of this result is obtained in \cite[Proposition 3.3]{zapata2016eberlein}. 
We have the following statement. 
\begin{proposition}
Let $(E,\mathscr{T})$ be a locally $L^0$-convex module and $S\subset E$ stable $L^0$-convex.
Then it holds 
\begin{align*}
\text{cl}(S)=\text{cl}_{\sigma(E,E^\ast_s)}(S)=\text{cl}_{\sigma(E,E^\ast_{\epsilon,\lambda})}(S)=\text{cl}_{\sigma(E,E^\ast_0)}(S). 
\end{align*}
\end{proposition}
\subsection{Stable dual pairs}
Given a stable dual pair $(E,F,\langle \cdot,\cdot\rangle)$, one has $(E,\sigma_0(E,F))^\ast=F$ by \cite[Theorem 3.4]{guo2015random2}. 
For a conditional dual pair $(\mathbf{E},\mathbf{F},\langle \cdot,\cdot\rangle)$, it holds $\mathbf{(E,\sigma(E,F))^\ast}=\mathbf{F}$ by \cite[Corollary 4.48]{martindiss}. 
We can derive from the latter the following extension of \cite[Theorem 3.4]{guo2015random2}. 
\begin{proposition}
Let  $(E,F,\langle \cdot,\cdot\rangle)$ be a stable dual pair. Then 
$$(E,\sigma_0(E,F))^\ast=(E,\sigma_s(E,F))^\ast=(E,\sigma_{\epsilon,\lambda}(E,F))^\ast=F.$$
\end{proposition}  
\begin{proof}
Since the induced conditional structure $(\mathbf{E},\mathbf{F},\langle \cdot,\cdot\rangle)$ is a conditional dual pair, by \cite[Corollary 4.48]{martindiss}, one has $(E,\sigma_s(E,F))^\ast=F$. 
The claim thus follows from  \cite[Theorem 3.4]{guo2015random2}.   
\end{proof}
Next, by an adaptation of the proof of the conditional Banach-Alaoglu theorem (see \cite[Theorem 5.10]{DJKK13}), we can derive the following module variant of it. 
\begin{theorem}\label{t:alaoglu}
Let $(E,\mathscr{T})$ be a stable locally $L^0$-convex module. 
Then 
\[
U^\circ:=\left\{ f\in E^\ast  \colon  f(x)\leq 1\text{ for all }x\in U\right\} 
\]
 is  stable  $\sigma_s(E^\ast,E)$-compact.
\end{theorem}
\begin{corollary}
Let $(E,\Vert\cdot\Vert)$ be an $L^0$-normed module. 
Then $\{ f\in E^\ast \colon \Vert f\Vert\leq 1 \}$ is  stable  $\sigma_s(E^\ast,E)$-compact. 
\end{corollary}
For a stable set $S$, its \emph{stable $L^0$-convex hull} $\text{co}_s(S)$ is defined as the collection of all  stable  finite sums $\sum_{\1\leq\n\leq\m} r_\n x_\n$, where $(x_\n)_{\1\leq\n\leq\m}\subset S$ is  stable  finite and $(r_\n)_{\1\leq\n\leq\m}\subset L^0$ is  stable  finite with  $r_\n\geq 0$ and $ \sum_{\1\leq\n\leq \m} r_\n=1$. 
A bipolar theorem for stable dual pairs w.r.t.~the $L^0$-topology is proved in \cite[Theorem 3.4]{guo2015random2}. 
Its conditional version is provided in \cite[Theorem 5.9]{DJKK13} from which we obtain: 
\begin{proposition}  
Let $(E,F,\langle \cdot,\cdot\rangle)$ be a stable dual pair and $S\subset E$ stable. 
Then 
$$S^{\circ\circ}=\text{cl}(\text{co}_s(S\cup\{0\})).$$
\end{proposition}
\subsection{$L^0$-normed modules}
Given $L^0$-normed modules $(E_1, \Vert\cdot\Vert_1)$ and $(E_2,\Vert\cdot\Vert_2)$, we denote by $\mathscr{L}(E_1,E_2)$ the set of all continuous $L^0$-linear functions $f\colon E_1 \to E_2$.  
Then $\mathscr{L}(E_1,E_2)$ has the structure of a stable $L^0$-module. 
Define the \emph{$L^0$-operator norm} by 
\[
\Vert f\Vert:=\sup\{ \Vert f(x)\Vert_2 \colon x\in E_1, \Vert x\Vert_1\leq 1\}. 
\]
We have the following version of the Krein-\v{S}mulian theorem, which can be proved by an adaptation of its conditional version \cite[Theorem 5.12]{DJKK13}.  
\begin{theorem} Let $(E,\Vert\cdot\Vert)$ be a  stable  complete $L^0$-normed module and $S\subset E^\ast$ stable $L^0$-convex. 
Then $S$ is $\sigma(E^\ast,E)$-closed if and only if $S\cap \left\{f\in E^\ast \colon \Vert f\Vert\leq r\right\}$ is $\sigma(E^\ast,E)$-closed for each $r\in L^0_{++}$. 
 \end{theorem}
The conditional Eberlein-\v{S}mulian Theorem \cite[Theorem 4.7]{zapata2016eberlein} provides the following module variant. 
\begin{theorem}
Let $(E,\Vert\cdot\Vert)$ be an $L^0$-normed and $S\subset E$ stable. 
Then the following are equivalent. 
\begin{itemize}
\item[(i)]  $S$ is  stable  $\sigma_s(E,E^\ast)$-compact. 
\item[(ii)] For every stable sequence $(x_\n)\subset S$ there exists a stable subsequence $(x_{\n_k})$ which converges in $S$ w.r.t.~the $\sigma_0(E,E^\ast)$-topology. 
\item[(iii)] For every stable sequence $(x_\n)\subset S$ there exists a stable subsequence $(x_{\n_k})$ which converges in $S$  w.r.t.~the $\sigma_s(E,E^\ast)$-topology. 
\end{itemize} 
\end{theorem}

\begin{proof}
The equivalence $(i)\Leftrightarrow (iii)$ follows from a rountine adaptation of the proof of \cite[Theorem 4.7]{zapata2016eberlein}.
$(iii)\Rightarrow(ii)$ is trivially satisfied since $\sigma_s(E,E^\ast)$ is finer than $\sigma_0(E,E^\ast)$. 

$(ii)\Rightarrow(iii)$: Let $(x_\n)$ be a stable sequence in $S$ and $(x_{\n_k})$ a stable subsequence which converges to $x\in S$ w.r.t.~the $\sigma_0(E,E^\ast)$-topology. 
Any $U\in\mathscr{U}_s(x)$ is of the form $U=\sum_k 1_{A_k}U_k$ with $(a_k)\in p(1)$ and $U_k\in \mathscr{U}_0(x)$ for all $k\in\N$. 
For each $k\in\N$, choose $\m_k\in L^0_s(\N)$ such that $x_{\n_{k}}\in U_k$ for all $k\geq\m_k$. 
Let $\m:=\sum_k 1_{A_k}\m_k$. 
Then, $x_{\n_{k}}\in U$ for each $k\geq\m$.  
\end{proof}

Finally, a transcription of \cite[Theorem 2.5]{OZ2017stabil} yields the following module version of a variant of the James compactness theorem. 

\begin{theorem}
Let $(E,\Vert\cdot\Vert)$ be a  stable  complete $L^0$-normed module such that $\{f\in E^\ast\colon \|f\|\leq 1\}$ is  stable  sequentially compact. 
Let $S\subset E$ be stable and $\sigma_s(E,E^\ast)$-closed. 
Then $S$ is  stable  $\sigma_s(E,E^\ast)$-compact if and only if for each $f\in E^\ast$ there exists $y\in E$ such that $f(y)= \sup_{x\in S} f(x)$.
\end{theorem}
\subsection{Stable metric spaces}
The following statement is a translation of the conditional version of the Baire category theorem (see \cite[Theorem 3.3]{zapata2016eberlein}). 
\begin{theorem}
Let $(S,d)$ be a  stable  complete metric space. 
Suppose that $(E_\n)$ is a stable sequence of stable closed subsets of $S$ with $S=\cup E_\n$. 
Then there exist $x\in S$, $r\in L^0_{++}$ and $\n\in L^0_s(\N)$ such that $B_r(x)\subset E_\n$.
\end{theorem}
We have the following module variant of the uniform boundedness principle as a consequence of \cite[Theorem 3.4]{zapata2016eberlein}. 
\begin{theorem}
Let $(E_1, \Vert\cdot\Vert_1)$ and $(E_2,\Vert\cdot\Vert_2)$ be $L^0$-normed modules and $E_1$  stable  complete. 
Suppose that $S$ is a stable subset of $\mathscr{L}(E_1,E_2)$ such that for every $x\in E_1$ there exists $r_x\in L^0$ with $\Vert f(x)\Vert_2 \leq r_x$ for all $f\in S$. 
Then there exists $r\in L^0_{++}$ such that $\Vert f\Vert\leq r$ for all $f\in S$.
\end{theorem}
Next, we extend Banach's fixed point theorem to $L^0$-modules. 
\begin{theorem}
Let $(S,d)$ be a  stable  complete metric space and $T:S\rightarrow S$ a stable continuous function such that there exists $r\in L^0$ with $0\leq r< 1$ and such that 
$$d(T(x),T(y))\leq r d(x,y)\quad\text{ for all } x,y\in S.$$ 
Then there exists a unique $z\in S$ satisfying $T(z)=z$.
Furthermore, for any $x_1\in S$, the stable sequence $(x_\n)$ defined by $x_\n:=\sum_{k\in\N} 1_{\{\n=k\}}x_k$, where $x_{n+1}:=T(x_n)$ for $n\geq 2$, converges to $z$. 
\end{theorem}
\begin{proof}
Our proof is based on a simple proof of the classical contraction principle given in \cite{palais2007simple}. 
For $\n\in L^0_s(\N)$ with $n=\sum n_k|a_k$, we define
\[
T^n(x):=\sum T^{n_k}(x)|a_k,\quad s\in S. 
\]
From stability we obtain
\begin{equation}
\label{eq: contraction}
d({T}^\n(x_1),T^\n(x_2))\leq r^\n d(x_1,x_2)\quad\text{ for all }x_1,x_2\in S.	
\end{equation}
Now for fixed $x_1,x_2\in S$, by triangle inequality, 
\[
d(x_1,x_2)\leq d(x_1,T(x_1))+d(T(x_1),T(x_2))+d(T(x_2),x_2), 
\]
from which it follows that 
\begin{equation}
\label{eq: contractionII}
d(x_1,x_2)\leq \frac{1}{1-r}(d(x_1,T(x_1))+d(x_2,T(x_2)))\quad\text{ for all }x_1,x_2\in S.
\end{equation}
As for uniqueness, if $z_1,z_2$ satisfy $z_1=T(z_1)$ and $z_2=T(z_2)$, then from \eqref{eq: contractionII} one has $d(z_1,z_2)=0$.
Now fix $x\in S$. 
For $\n,\m\in L^0_s(\N)$, if we replace $x_1$ and $x_2$ by $T^\n(x)$ and $T^\m(x)$ in \eqref{eq: contractionII}, then we obtain from \eqref{eq: contraction}
\begin{align*}
d(T^\n(x),T^\m(x))&\leq \frac{1}{1-r}(d(T^\n(x),T^\n(T(x)))+d(T^\m(x),T^\m(T(x))))\\
&\leq \frac{r^\n+r^\m}{1-r}d(x,T(x)). 
\end{align*}
Since $r<1$, we obtain $\lim_{\n,\m}d(T^\n(x),T^\m(x))=0$. 
Since $S$ is conditional complete, the conditional sequence $(T^\n(x))$ has a conditional limit $z\in S$. 
By letting first $\m$ and second $\n$ tend to $\infty$, we have $T(z)=z$. 
\end{proof}
Finally, we prove a conditional version of the \`{A}rzela-Ascoli theorem which is based on an adaptation of a proof of the classical result in \cite{hanche2009kolmogorov}. 
Let $(S,\mathscr{T})$ be a stable topological space. 
Denote by $\mathscr{C}(S,L^0)$ the set of all stable continuous functions $f\colon S \to L^0$.  
Notice that $\mathscr{C}(S,L^0)$ is a stable $L^0$-module. 
By Theorem \ref{thm: lowSemCompact},  for  stable  compact $S$, 
$$d_\infty(f,g):=\max\left\{|f(x)-g(x)| \colon x\in S \right\},\quad f,g\in\mathscr{C}(S,L^0),$$  
is a well-defined stable metric.
\begin{definition}
A stable subset $M\subset \mathscr{C}(S,L^0)$ is  \emph{ stable  equicontinuous} if for every $x\in S$ and $r\in L^0_{++}$ there exists a  neighbourhood  $V$ of $x$ such that $|f(x)-f(y)|\leq r$ for all $y\in V$ and $f\in M$.
\end{definition}
Finally, we state a module variant of the Arzel\`a-Ascoli theorem. 
\begin{theorem}
Let $(S,\mathscr{T})$ be a  stable  compact topological space. 
A stable subset $M$ of $\mathscr{C}(S,L^0)$ is  stable  totally bounded in the stable metric $d_\infty$ if and only if it is  pointwise  stable  bounded and  stable  equicontinuous. 	
\end{theorem}
\begin{lemma}\label{lem: AA}
Let $(S,d)$ be a stable metric space. 
Suppose that for every $r\in L^0_{++}$ there exist $s_r\in L^0_{++}$, a stable metric space $(S_r,d_r)$ in some $L^0$-module $E_r$, and a stable function $f_r:S\rightarrow S_r$ so that $f_r(S)$ is  stable  totally bounded and for $x,y\in S$
\begin{equation}
\label{eq: ineMetr}
d_r(f_r(x),f_r(y))\leq s_r\text{ implies }d(x,y)\leq r.
\end{equation}
Then $S$ is  stable  totally bounded.
\end{lemma}
\begin{proof}
Let $r\in L^0_{++}$. 
Choose $s_r\in L^0_{++}$, a stable metric space $(S_r,d_r)$ and $f_r:S\rightarrow S_r$ according to the assumptions of the statement above. 
Since $f_r(S)$ is  stable  totally bounded, there exists a  stable  finite subset $N\subset S$ such that
 $f_r(S)=\underset{x\in S}\cup B_{s_r}(f_r(x))$. 
By (\refeq{eq: ineMetr}), we have $S=\underset{x\in S}\cup B_r(x)$. 
Thus $S$ is  stable  totally bounded.
\end{proof}

\begin{theorem}\label{t:arzela-ascoli}
Let $(K,\mathscr{T})$ be a  stable  compact topological space. 
Then a stable subset $M$ of $\mathscr{C}(K,L^0)$ is  stable  totally bounded with respect to $d_\infty$ if and only if  
\begin{itemize}
\item[(i)] $M_x:=\{ f(x) \colon f\in M\}$ is  stable  bounded for each $x\in K$, and
\item[(ii)] $M$ is conditional equicontinuous.
\end{itemize}
\end{theorem}
\begin{proof}
Suppose that $M$ is  stable  equicontinuous and $M_x$ is  stable  bounded for all $x\in K$.  
Using that  $K$ is  stable  compact and $M$ is  stable  equicontinuous, one can find a  stable  finite sequence $(x_\n)_{1\leq\n\leq\m}\subset K$, a stable  family $(V_\n)_{1\leq\n\leq\m}$, where $V_\n$ is a  neighbourhood  of $x_\n$ for all $1\leq\n\leq\m$, such that $K=\cup_{1\leq\n\leq\m} V_\n$ and $|f(x_\n)-f(y)|\leq r$ whenever $1\leq\n\leq\m$, $y\in V_\n$ and $f\in M$. 

Let $g:M\rightarrow (L^0)^\m$ be defined by $g(f):=(f(x_\n))_{1\leq\n\leq\m}$, which is a stable function. 
Note that $\Vert (r_\n)_{1\leq\n\leq\m}\Vert_\infty:=\underset{1\leq\n\leq\m}\max|r_\n|$ defines an $L^0$-norm on $(L^0)^\m$. 

Since $M_{x_\n}$ is  stable  bounded for each $x\in K$, we can choose $r_\n\in L^0_{++}$ such that $|f(x_\n)|<r_\n$ for all $f\in M$. Then, one has that  
$$\Vert g(f)\Vert_\infty:= \underset{1\leq\n\leq\m}{\max} |f(x_\n)|\leq \underset{1\leq\n\leq\m}\max r_\n<+\infty\quad\text{ for all }f\in M.$$
 
 This means that $g(M)$ is  stable  bounded in $((L^0)^\m,\Vert\cdot\Vert_\infty)$. 

Now, suppose that $f_1,f_2\in M$ with $\Vert g(f_1)-g(f_2) \Vert_\infty\leq r$ and $y\in K$. 
Since $K$ is covered by $(V_\n)_{1\leq\n\leq\m}$, we can pick $\n$, $1\leq\n\leq\m$, such that $y\in V_\n$. 
Then, one has
\[
|f_1(y)-f_2(y)|\leq |f_1(y)-f_1(x_\n)|+|f_1(x_\n)-f_2(x_\n)|+|f_2(x_\n) - f_2(z)|\leq 3 r.
\]

Since $y\in K$ is arbitrary, this implies that $d_\infty(f_1,f_2)\leq 3 r$. 
Lemma \ref{lem: AA} yields the result.

Conversely, suppose that $M\subset \mathscr{C}(K,L^0)$ is  stable  totally bounded. 
In particular, it is  stable  bounded and, therefore, $M_x$ is  stable  bounded for each $x\in K$. 

Let us show that $M$ is  stable  equicontinuous. 
Indeed, for fixed $x\in K$ and $r\in L^0_{++}$, since $M$ is  stable  totally bounded, one can pick a  stable  finite sequence $(f_\n)_{1\leq \n\leq \m}\subset M$ so that $M=\underset{1\leq \n\leq \m}\cup B_r(f_\n)$. 

For any $n\in\N$, we define $h_n:=1_{A_n}f_n$ 
where $a_n:=\vee\left\{b\in\mathcal{A} \colon 1_B n\leq 1_B\m\right\}$. 
For each $n\in\N$, since $h_n$ is continuous at $x$, we can take a  neighbourhood  $V_n$ of $x$ so that $|h_n(h)-h_n(y)|\leq r$ for all $y\in V_n$. 
Now, for any $\n\leq \m$, which is of the form $n=\sum n_k|b_k$, we define $V_\n=\sum 1_{B_k}V_{n_k}$. 
Then $(V_\n)_{1\leq\n\leq \m}$ is a stable family of  neighbourhood s of $x$ such that
\[
|h_\n(x)-h_\n(y)|=|f_\n(x)-f_\n(y)|\leq r\quad\text{ for all }y\in V_\n,\:\n\leq \m.
\] 
Let $V:=\underset{1\leq\n\leq \m}\cup V_\n$. 
If $f\in M$, one has that $f\in B_r(f_\n)$ for some $\n\leq \m$. 
Thus, for any $y\in V$ it holds
\[
|f(y)-f(x)|\leq |f(y)-f_\n(x)|+|f_\n(x)-f_\n(y)|+|f_\n(y)-f(x)|\leq 3 r.
\] 
This proves that $M$ is  stable  equicontinuous. 
\end{proof}

\end{document}